\documentclass[oneside,11pt]{amsart}

\usepackage{amsmath, mathrsfs, bbm}
\usepackage{graphicx}
\usepackage[T1]{fontenc}

\usepackage[small]{caption}

\usepackage{tikz-cd}
\usetikzlibrary{arrows,chains,matrix,positioning,scopes}

\makeatletter
\tikzset{join/.code=\tikzset{after node path={%
\ifx\tikzchainprevious\pgfutil@empty\else(\tikzchainprevious)%
edge[every join]#1(\tikzchaincurrent)\fi}}}
\makeatother
\tikzset{>=stealth',every on chain/.append style={join},
         every join/.style={->}}
\tikzstyle{labeled}=[execute at begin node=$\scriptstyle,
   execute at end node=$]

\usepackage{amsfonts}
\usepackage{amssymb}
\usepackage{amsxtra}

\usepackage[all]{xy}
\SelectTips{cm}{12}
\CompileMatrices

\usepackage{ifthen}

\newcommand{\showcomments}{yes}

\newsavebox{\commentbox}
{\ifthenelse{\equal{\showcomments}{yes}}%
{\footnotemark
    \begin{lrbox}{\commentbox}
    \begin{minipage}[t]{1.25in}\raggedright\sffamily\upshape\tiny
    \footnotemark[\arabic{footnote}]}
{\begin{lrbox}{\commentbox}}}%
{\ifthenelse{\equal{\showcomments}{yes}}%
{\end{minipage}\end{lrbox}\marginpar{\usebox{\commentbox}}}
{\end{lrbox}}}

\theoremstyle{plain}
\newtheorem{theorem}{Theorem}[section]
\newtheorem{corollary}[theorem]{Corollary}
\newtheorem{lemma}[theorem]{Lemma}
\newtheorem{proposition}[theorem]{Proposition}

\theoremstyle{definition}
\newtheorem{defn}[theorem]{Definition}
\newtheorem{example}[theorem]{Example}

\newcommand{\bbp}{\mathbb{P}}
\newcommand{\bndry}{\partial}

\newcommand{\R}{\mathbb{R}}

\newcommand{\relbndry}{\bndry (G,\bbp)}

\newcommand{\boundary}{\partial}

\DeclareMathOperator{\stab}{Stab}
\DeclareMathOperator{\Stab}{Stab}

\DeclareMathOperator{\val}{val}

\DeclareMathOperator{\QI}{QI}

\newcommand{\set}[2]{\{\,{#1} \mid {#2} \,\}}
\newcommand{\bigset}[2]{ \bigl\{ \, {#1} \bigm| {#2} \, \bigr\} }

\newcommand{\abs}[1]{\left|{#1}\right|}
\renewcommand{\emptyset}{\varnothing}
\renewcommand{\setminus}{-}

\newcommand{\field}[1]{\mathbb{#1}}
\newcommand{\Z}{\field{Z}}

\newcommand{\Hyp}{\field{H}}

\usepackage{bbm}
\newcommand{\id}{\mathbbm{1}}

\renewcommand{\P}{\field{P}}

\newcommand{\of}{\circ}

\usepackage{combelow}
\newcommand{\Drutu}{Dru{\cb{t}}u}

\begin{document}
\title[On canonical splittings of relatively hyperbolic groups]{On canonical splittings of relatively hyperbolic groups}

\author{Matthew Haulmark}
\address{Department of Mathematical Sciences\\
Binghamton University\\
P.O. Box 6000\\
Binghamton, NY
13902\\ 
USA}
\email{haulmark@binghamton.edu}

\author{G. Christopher Hruska}
\address{Department of Mathematical Sciences\\
University of Wisconsin--Milwaukee\\
P.O. Box 413\\
Milwaukee, WI 53201\\
USA}
\email{chruska@uwm.edu}

\date{\today}

\begin{abstract}
A JSJ decomposition of a group is a splitting that allows one to classify all possible splittings of the group over a certain family of edge groups.
Although JSJ decompositions are not unique in general, Guirardel--Levitt have constructed a canonical JSJ decomposition, the tree of cylinders, which classifies splittings of relatively hyperbolic groups over elementary subgroups.

In this paper, we give a new topological construction of the Guirardel--Levitt tree of cylinders, and we show that this tree depends only on the homeomorphism type of the Bowditch boundary. Furthermore the tree of cylinders admits a natural action by the group of homeomorphisms of the boundary.
In particular, the quasi-isometry group of $(G,\mathbb{P})$ acts naturally on the tree of cylinders.
\end{abstract}
\maketitle

\section{Introduction}

A word hyperbolic group $G$ has a natural Gromov boundary at infinity $\boundary G$, whose topological properties are often closely related to properties of $G$.
A particularly strong theorem of this nature states that $\boundary G$ is homeomorphic to a circle if and only if $G$ is a cocompact Fuchsian group \cite{T88,Gabai92,CJ94}.
Bowditch used this result to prove that each hyperbolic group with connected boundary has a canonical JSJ tree for splittings over $2$--ended subgroups \cite{Bow98_JSJ} that depends only on the topology of $\boundary G$. 
Bowditch's topological JSJ theorem is a key ingredient in Ha\"{i}ssinsky's proof of the following: if $G$ is a word hyperbolic group with boundary homeomorphic to the limit set of a convex cocompact Kleinian group that contains no embedded Sierpi\'{n}ski carpet, then $G$ is virtually a convex cocompact Keinian group \cite{Haissinsky17}.

According to Gromov \cite{Gro87}, quasi-isometric hyperbolic groups always have homeomorphic boundaries.  However ``homeomorphic boundary'' equivalence is broader in general than quasi-isometric equivalence (see, for instance, \cite{Bourdon97,CashenMartin17}).  Understanding the difference between these two notions of equivalence is a central problem in geometric group theory.
For example, the Cannon Conjecture is equivalent to the statement that every group with boundary homeomorphic to the $2$--sphere is quasi-isometric to $\Hyp^3$.
Malone, Cashen--Martin, and Dani--Thomas have used the topological invariance of Bowditch's JSJ decomposition to understand the relation between quasi-isometric equivalence and homeomorphic boundary equivalence for certain families of hyperbolic groups \cite{Malone_Thesis,CashenMartin17,DaniThomas_JSJ}.

It is natural to ask whether Bowditch's canonical JSJ decomposition can be generalized from hyperbolic groups to relatively hyperbolic groups.
Several partial generalizations along these lines are known.
For finitely presented groups (with no hyperbolicity requirement) Papasoglu shows that the JSJ decomposition for splittings over $2$--ended groups is invariant under quasi-isometries \cite{Papasoglu_Quasi-Lines}. However, the proof in \cite{Papasoglu_Quasi-Lines} depends on a theorem about separating quasilines that need not hold for non--finitely presented groups (see, for example, \cite{Papasoglu12_Lamplighter}). Thus Papasoglu's splitting theorem does not apply to arbitrary relatively hyperbolic groups.
Even in the finitely presented case, it is not clear from Papasoglu's methods whether this decomposition is determined by the topology of the boundary.

Alternatively Guirardel--Levitt have introduced a canonical JSJ decomposition, called the tree of cylinders, for splittings of a relatively hyperbolic group $G$ over elementary subgroups relative to peripheral subgroups \cite{GL_TreesCyl}, which is invariant under automorphisms of $G$.
In the special case in which $G$ is a free group and $\mathbb{P}$ is a family of maximal cyclic subgroups, 
Cashen \cite{Cashen16_JSJ} (building on work of Otal \cite{Otal92})
shows that the JSJ tree of cylinders is determined by the topology of the boundary.
However Cashen's proof appears to be quite specialized to the case of free groups.

Groff claims in \cite{Grof13} that the JSJ tree of cylinders can be constructed from the topology of the boundary $\boundary(G,\mathbb{P})$ using a construction suggested by Papasoglu--Swenson in \cite{PS06}.
Unfortunately the result claimed in \cite[Thm.~4.6]{Grof13} is false; a counterexample to this claim may be found in Section~\ref{sec:Levitt}.  The existence of this counterexample indicates that the suggested strategy does not produce a canonical $G$--tree in general.

Bowditch and Guralnik \cite{Bow98_JSJ,Gur} have indicated that exact cut pairs are more amenable to detailed analysis than arbitrary cut pairs.
See Section~\ref{subsec:CutPoints} for the definition of exactness.
For a more detailed investigation of non-exact cut pairs, we refer the reader to Section~\ref{sec:Levitt} and also to \cite{HruskaWalsh_Parabolic}.
The canonical JSJ tree introduced in the following theorem may be described topologically as the simplicial tree dual to the family of all cut points and inseparable exact cut pairs and is constructed in Theorem~\ref{Thm: construction of T}.

\begin{theorem}
\label{thm: main thm}
Let $(G,\mathbb{P})$ be a relatively hyperbolic group.  Suppose the Bowditch boundary $M=\boundary(G,\mathbb{P})$ is connected.
The canonical JSJ tree of cylinders for splittings of $G$ over elementary subgroups relative to peripheral subgroups is a tree $T(M)$ that depends only on the topological structure of the boundary $M$.
\end{theorem}

By a theorem of Bowditch, the boundary $\boundary(G,\mathbb{P})$ is connected precisely when each subgroup $P\in \mathbb{P}$ is infinite and $G$ is one ended relative to $\mathbb{P}$.
In this article, we only consider relatively hyperbolic pairs $(G,\mathbb{P})$ such that $G$ is finitely generated.
In that context, if the boundary $\boundary(G,\mathbb{P})$ is connected then it is locally connected by work of Bowditch in a special case and Dasgupta in the general case \cite{Bow01_Peripheral,Dasgupta_Thesis}.
(For the original hyperbolic case, see \cite{Swa}.)
We note that local connectivity is used extensively throughout the proof of Theorem~\ref{thm: main thm}, as is also the case for the main results of \cite{Bow98_JSJ}.

Since the tree in Theorem~\ref{thm: main thm} is defined purely in terms of the topology of $M$, it is invariant under homeomorphisms in the following sense.

\begin{corollary}
\label{cor:Functorial}
Suppose $M_1,M_2$ are connected Bowditch boundaries of relatively hyperbolic pairs $(G_1,\mathbb{P}_1)$ and $(G_2,\mathbb{P}_2)$.
Let $T(M_i)$ denote the JSJ tree corresponding to $M_i$ as given by Theorem~\ref{Thm: construction of T}.
Each homeomorphism $f\colon M_1\to M_2$ induces a graph isomorphism $\hat{f}\colon T(M_1) \to T(M_2)$.
If $f_i\colon M_i \to M_{i+1}$ are homeomorphisms for $i=1,2$ then $\widehat{f_2f_1} = \hat{f_2}\hat{f_1}$, and also $\widehat\id = \id$.
\end{corollary}

Each vertex of the JSJ tree has a ``type'' in the usual sense of JSJ theory.  In the present setting the possible vertex types are parabolic, $2$--ended, quadratically hanging, and rigid (see Proposition~\ref{prop: vertex types}).  As in \cite{CashenMartin17} the induced automorphism $\hat{f}$ of trees is type preserving and induces homeomorphisms between the limits sets of corresponding vertex groups relative to the family of limit sets of adjacent edge groups.

Let $(G,\mathbb{P})$ and $(G',\mathbb{P}')$ be relatively hyperbolic with finite generating sets $A$ and $A'$.  
A \emph{coarsely Lipschitz map of pairs} $(G,\mathbb{P}) \to (G',\mathbb{P}')$ is a coarsely Lipschitz map $G \to G'$ such that for some $R<\infty$ each peripheral coset $gP$ with $P \in \mathbb{P}$ has image contained in an $R$--tubular neighborhood of some peripheral coset $g'P'$ with respect to the word metric $d_{A'}$.
A \emph{quasi-isometry of pairs} $(G,\mathbb{P}) \to (G',\mathbb{P}')$ is a coarsely Lipschitz map of pairs admitting a quasi-inverse that is also a coarsely Lipschitz map of pairs $(G',\mathbb{P}') \to (G,\mathbb{P})$.
A result of Behrstock--\Drutu--Mosher states that if each $P \in \mathbb{P}$ and each $P' \in \mathbb{P}'$ are non--relatively hyperbolic, then every quasi-isometry $G \to G'$ is a quasi-isometry of pairs \cite[Thm.~4.1]{BehrstockDrutuMosher_Thick} (see also \cite{Schwartz95}).

By \cite{Grof13}, any quasi-isometry $(G,\mathbb{P}) \to (G',\mathbb{P}')$ induces a homeomorphism $\boundary(G,\mathbb{P}) \to \boundary(G',\mathbb{P}')$.
(Groff's proof of this quasi-isometry theorem is modelled on a construction of Schwartz for rank one symmetric spaces \cite{Schwartz95} and does not involve the mistake mentioned above.)

\begin{corollary}
\label{cor:QIInvariant}
Let $(G,\mathbb{P})$ and $(G',\mathbb{P}')$ be relatively hyperbolic with connected boundaries.  Every quasi-isometry of pairs $(G,\mathbb{P}) \to (G',\mathbb{P}')$ induces a vertex-type preserving isomorphism of JSJ trees $T_G \to T_{G'}$.
\end{corollary}

The boundary homeomorphism invariance given by Corollary~\ref{cor:Functorial} and the quasi-isometry invariance given by Corollary~\ref{cor:QIInvariant} could potentially be useful tools for attacking classification problems in families of relatively hyperbolic groups and for better understanding the difference between these two notions of equivalence.


If $(G,\mathbb{P})$ is relatively hyperbolic with connected boundary then the relative quasi-isometry group $\QI(G,\mathbb{P})$ naturally acts on the canonical JSJ tree for $(G,\mathbb{P})$.  The \emph{relative quasi-isometry group} is the group of all self quasi-isometries of pairs modulo those that have finite distance from the identity in the sup-norm.
In \cite[\S 8]{Grof13}, Groff indicates the beginnings of a study of $\QI(G,\mathbb{P})$ via its action on the canonical JSJ tree.

\subsection{Overview}
Section~\ref{sec: Preliminaries} contains some background on JSJ decompositions of groups and the notions of cut points and exact cut pairs in Peano continua. In Section~\ref{sec:RelativelyHyperbolic} we review the structure of relatively hyperbolic groups and their boundaries, two equivalent definitions of a relatively quasiconvex subgroup, and the intersection properties of limit sets of relatively quasiconvex subgroups.  In Section~\ref{sec:NoCutPoint} we focus on the special setting of Bowditch boundaries that have no global cut point.  In this setting, we determine the structure of all local cut points of the boundary according to their valence and their groupings into various types of exact cut pair.
A key result proved in this section is Proposition~\ref{prop: only finitely many}, which states that the inseparable exact cut pairs of the boundary lie in only finitely many $G$--orbits.

The main result of Section~\ref{sec: trees separate} is Proposition~\ref{prop: rel qc}, which states that whenever a relatively hyperbolic group splits with relatively quasiconvex edge groups, each vertex group is also relatively quasiconvex.
A similar result is claimed in \cite{BigdelyWise13}.
In Section~\ref{sec:Decompositions} we study arbitrary actions of a relatively hyperbolic group $G$ on trees with elementary edge stabilizers such that each peripheral subgroup fixes a vertex.
We show that any such tree induces a decomposition of the boundary into halfspaces, corresponding to the halfspaces of the tree.

In Section~\ref{sec:Levitt} we discuss two mistakes related to \cite[Thm.~4.6]{Grof13}, the first of which was discovered by the second author and Genevieve Walsh in work related to \cite{HruskaWalsh_Parabolic}. The other mistake was communicated to the authors by Brian Bowditch.
Next in Section~\ref{sec: simplicial}, we use Bowditch's theory of peripheral splittings \cite{Bow01_Peripheral} to construct a simplicial tree $T$ dual to the nested family of all cut points and inseparable exact cut pairs of $\relbndry$. The tree $T$ records splittings of $(G,\bbp)$ relative to $\bbp$ over elementary subgroups. We also characterize the four types of vertex stabilizers of $T$.
Lastly, in Section~\ref{sec: last section} we show that this dual tree $T$ is a JSJ tree and non-elementary flexible vertex stabilizers of $T$ are quadratically hanging with finite fiber.

\subsection{Acknowledgements}
The authors would like to thank Brian Bowditch, Daniel Groves, and Jason Manning for helpful discussions regarding JSJ decompositions, Rips theory, and relative hyperbolicity.
The authors are grateful to Genevieve Walsh for encouraging them to explore the ideas that led to this article.
The authors are indebted to Dani Wise for explaining an argument for the proof of Proposition~\ref{prop: rel qc}.  Although the proof presented here is a bit different from Wise's suggestion, it is inspired by that discussion.
The authors also thank Wenyuan Yang and the anonymous referee for their helpful feedback.

This work was partially supported by a grant from the Simons Foundation (\#318815 to G. Christopher Hruska).

\section{Preliminaries}
\label{sec: Preliminaries}

\subsection{JSJ decompositions}
\label{sec:JSJ}

In this section, we review the notion of JSJ decompositions of groups.  A decomposition of a group as the fundamental group of a graph of groups is equivalent to an action of the group on a simplicial tree without inversions \cite{Serre}.  
We largely follow the Guirardel--Levitt formalism for JSJ decompositions, in which a JSJ tree is determined by certain universal properties among a given family of $G$--actions on trees.
We refer the reader to Guirardel--Levitt \cite{GL_JSJ} for a thorough examination of the concepts briefly reviewed below.

Let $G$ be a finitely generated group acting on a tree $T$ without inversions. 
We assume that all trees considered in this section are \emph{minimal} in the sense that there is no proper invariant subtree.

A subgroup $H<G$ acts \emph{elliptically} on $T$ if it fixes a point of $T$. If $\mathbb{E}$ is a collection of subgroups of $G$ that is closed under conjugation and passing to subgroups, then $T$ is an \emph{$\mathbb{E}$--tree} if every edge stabilizer is a member of $\mathbb{E}$.
If $\P$ is an arbitrary family of subgroups of $G$, an $(\mathbb{E},\P)$--{\it tree} is an $\mathbb{E}$--tree $T$ such that every $P\in\P$ acts elliptically on $T$.  An $(\mathbb{E},\P)$--tree is {\it universally elliptic} if its edge stabilizers act elliptically on every $(\mathbb{E},\P)$--tree. 
If $G$ acts on trees $T$ and $T'$, then $T$ \emph{dominates} $T'$ if there is a $G$--equivariant map $T\to T'$, or equivalently if each vertex stabilizer of $T$ also stabilizes a vertex of $T'$. 
Two $(\mathbb{E},\P)$--trees $T$ and $T'$ are \emph{equivalent} if $T$ dominates $T'$ and $T'$ dominates $T$.

\begin{defn}
An $(\mathbb{E},\P)$--tree $T$ is a \emph{JSJ tree for splittings of $G$ over $\mathbb{E}$ relative to $\P$} if it satisfies the following:
\begin{enumerate}
    \item $T$ is universally elliptic among all $(\mathbb{E},\P)$--trees.
    \item $T$ dominates any other universally elliptic $(\mathbb{E},\P)$--tree.
\end{enumerate}
We note that JSJ trees, when they exist, are typically not unique.  However any two JSJ $(\mathbb{E},\P)$--trees are always equivalent in the above sense. A vertex stabilizer $G_v$ of a JSJ tree over $\mathbb{E}$ relative to $\P$ is \emph{flexible} if there is another $(\mathbb{E},\P)$--tree on which $G_v$ does not act elliptically. 
\end{defn}

\begin{defn}[Quadratically hanging]
A vertex stabilizer $G_v$ of an $(\mathbb{E},\P)$--tree is \emph{quadratically hanging} if it is an extension
\[
   1\rightarrow F\rightarrow G_v\rightarrow \pi_1(\Sigma)\rightarrow 1,
\]
where $\Sigma$ is a compact hyperbolic two-orbifold and $F$ is an arbitrary group called the {\it fiber}. Additionally, it is required that each incident edge stabilizer and each group $G_v\cap gPg^{-1}$ for $P\in\P$ has image in $\pi_1(\Sigma)$ that is either finite or contained in a boundary subgroup of $\pi_1(\Sigma)$.
\end{defn}

\begin{defn}
\label{def: acylindrical}
A tree $T$ on which $G$ acts is $(k,C)$--{\it acylindrical} if the pointwise stabilizer of every arc of length $\geq k+1$ is of order $\leq C$. We say that $T$ is {\it acylindrical} if $T$ is $(k,C)$--acylindrical for some $k$ and $C$.
\end{defn}

\begin{lemma}[\cite{GL_deformation}, \S 4]
\label{lem:acylindrical}
Suppose $G$ is a finitely generated group acting minimally on trees $T$ and $T'$.
If the actions of $T$ and $T'$ are equivalent, then $T$ is acylindrical if and only if $T'$ is acylindrical.
\end{lemma}

\begin{proof}
If $G$ is finitely generated and acts minimally on a tree $T$, then the action has only finitely many orbits of edges \cite[Cor.~3.8]{SW79}.
Assume $T$ and $T'$ are equivalent, that $T'$ is $(k',C)$--acylindrical,
and that $H\le G$ is a subgroup of order at least $C+1$.
Then there exists an equivariant quasi-isometry $T\to T'$ (see \cite[Thm.~3.8]{GL_deformation}).
By equivariance, the fixed point set of $H$ in $T$ maps into the fixed point set of $H$ in $T'$.
Since the subtree of $T'$ fixed by $H$ has diameter at most $k'$, the subtree of $T$ fixed by $H$ has diameter at most $k$, where $k$ depends on $k'$ in terms of the quasi-isometry constants.
It follows that $T$ is $(k,C)$--acylindrical.
\end{proof}

\subsection{Global and local cut points in Peano continua}
\label{subsec:CutPoints}

In this section we establish notation and terminology for various separating or locally separating sets in a Peano continuum, \emph{i.e.}, a compact, connected, locally connected, metrizable space. A \emph{\textup{(}global\textup{)} cut point} of $M$ is a point $\eta\in M$ such that $M\setminus\{\eta\}$ is disconnected.
A subset $C \subseteq M$ is a \emph{cyclic element} if $C$ consists of a single cut point or contains a non-cutpoint $p$ and all points $q$ that are not separated from $p$ by any cut point of $M$.
A cyclic element is \emph{nontrivial} if it contains at least two points.
Each Peano continuum is the union of its cyclic elements, and each pair of cyclic elements intersects in at most one point that is a cut point of $M$ (see Kuratowski \cite[\S 52]{Kuratowski_VolII}).


A \emph{cut pair} is a set of two distinct points $\{ \zeta,\xi \}\subset M$ such that $M \setminus \{\zeta,\xi\}$ is disconnected but neither $\zeta$ nor $\xi$ is a cut point of $M$.  The next result relates the structures of cut points and cut pairs in $M$.

\begin{lemma}
\label{lem:CutPair}
Every cut pair of $M$ is contained in a unique cyclic element.
\end{lemma}

\begin{proof}
If $\{\zeta,\xi\}$ is a cut pair
and $\eta$ is a cut point, the closure of any component of $M\setminus\{\zeta,\xi\}$ not containing
$\eta$ is a connected subset of $M\setminus\{\eta\}$ containing both $\zeta$ and $\xi$.
\end{proof}

A point $\zeta\in M$ is a \emph{local cut point} if either $\zeta$ is a global cut point or $M\setminus\{\zeta\}$ is connected and has more than one end.  If $\zeta$ is a local cut point but not a global cut point, we define the {\it valence}  $\val(\zeta)$ to be the number of ends of $M\setminus\{\zeta\}$.  
The local cut points of $M$ naturally fall into the following three families.
Let $M(2)$ be the set of all \emph{bivalent} points of $M$, \emph{i.e.}, points of valence two.
Let $M(3+)$ be the set of all \emph{multivalent} points in $M$, \emph{i.e.}, points of finite valence three or greater.
Similarly let $M(\infty)$ be the set of \emph{apeirovalent} points, \emph{i.e.}, those points with infinite valence.

The two points of a cut pair are always local cut points, but never global cut points.  A cut pair $\{\zeta,\xi\}$ in a Peano continuum $M$ can have only a finite number of complementary components in $M$ and $\zeta,\xi$ both lie in the closure of each such component (see \cite[Cor.~3.4]{Gur}).

A cut pair $\{\zeta,\xi\}$ is \emph{exact} if $\val(\zeta)=\val(\xi)= n$, where $n$ is the number of components of $M \setminus \{\zeta,\xi\}$. Since $M$ is locally connected, it follows that an exact cut pair must have finite valence (see \cite[Thm.~3-10]{HockingYoung_Topology}).
An exact cut pair is \emph{bivalent} or \emph{multivalent} depending on the valence of its points.

According to \cite[Lem.~3.8]{Bow98_JSJ} multivalent exact cut pairs are always disjoint.  In order to study the more intricate structure of bivalent exact cut pairs, Bowditch introduces an equivalence relation $\sim$ on $M(2)$, defined by $\zeta\sim \xi$ if either $\zeta=\xi$ or $\{\zeta,\xi\}$ is a bivalent exact cut pair \cite[\S 3]{Bow98_JSJ}.

The closure $\overline{N}$ in $M$ of a $\sim$--class $N$ containing at least three elements is a \emph{necklace}.  By \cite[Lem.~3.2]{Bow98_JSJ}, every necklace $\overline{N}$ is \emph{cyclically separating} in the following sense.
For each finite subset $F \subseteq \overline{N}$ there is a map $i\colon F \to S^1$ such that for all $a,b,c,d\in F$ the points $a$ and $c$ are separated by $\{b,d\}$ in $M$ if and only if $i(a)$ and $i(c)$ are separated by $\bigl\{i(b),i(d)\bigr\}$ in $S^1$.
A \emph{jump} in a necklace $\overline{N}$ is a pair of points $\{a,b\}$ that are not separated in $M$ by any pair of points $\{c,d\}$ of $\overline{N}$.

An exact cut pair is \emph{isolated} if it is either multivalent or equal to an entire $\sim$--class of bivalent points.
An exact cut pair $\{\zeta,\xi\}$ of $M$ is \emph{inseparable} if $\zeta$ and $\xi$ do not lie in distinct components of the complement of any other exact cut pair.

\begin{lemma}
\label{lemma: cut not separated by M(3+)}
If $M$ is a Peano continuum with no cut points, then the following hold.
\begin{enumerate}
    \item 
    \label{item:MultivalentInseparable}
    Multivalent exact cut pairs are inseparable.
    \item 
    \label{item:CutNotSepbyMulti} 
    A cut pair cannot be separated by a multivalent exact cut pair.
\end{enumerate}
\end{lemma}

\begin{proof}
By exactness, the complement $M\setminus\{\sigma,\tau\}$ has at least three components.
In order to show (\ref{item:MultivalentInseparable}), suppose $\{\zeta,\xi\}$ is another cut pair of $M$ (not necessarily exact).
It suffices to consider the case when $\{\sigma,\tau\}$ is disjoint from $\{\zeta,\xi\}$.
But then at least one component $C$ of $M\setminus\{\sigma,\tau\}$ does not intersect $\{\zeta,\xi\}$.  The closure $\overline{C}$ is a connected set in $M\setminus\{\zeta,\xi\}$ containing both $\sigma$ and $\tau$, so $\{\zeta,\xi\}$ does not separate $\{\sigma,\tau\}$, establishing (\ref{item:MultivalentInseparable}).

To show (\ref{item:CutNotSepbyMulti}), let $\{\sigma,\tau\}$ be a multivalent exact cut pair, and let $\{\zeta,\xi\}$ be any other cut pair.  As before we assume that the two cut pairs are disjoint.  By (\ref{item:MultivalentInseparable}) we know that $\sigma$ and $\tau$ are in the same component $C$ of $M \setminus \{\zeta,\xi\}$.  Since $\{\zeta,\xi\}$ is a cut pair, its complement contains at least one other component $C'$ that contains neither $\sigma$ nor $\tau$.
Thus the closure $\overline{C}'$ is a connected set in the complement of $\{\sigma,\tau\}$ containing both $\zeta$ and $\xi$, which gives (\ref{item:CutNotSepbyMulti}).
\end{proof}

\section{Relative hyperbolicity and relative quasiconvexity}
\label{sec:RelativelyHyperbolic}

In this section we review relatively hyperbolic groups, their boundaries, and their relatively quasiconvex subgroups.  We discuss several of their basic properties.

In \cite{CannonCooper92}, Cannon--Cooper introduce a key construction for understanding the geometry of a relatively hyperbolic group.  They construct a model of a horoball based on an arbitrary graph and then form a ``cusped space'' by gluing such horoballs onto the peripheral cosets in a Cayley graph. 
Bowditch introduced a slightly different model for this construction in \cite{BowditchRelHyp}.
Due to work of Groves--Manning and Groves--Manning--Sisto, either of these cusped space models may be used to characterize the notion of a relatively hyperbolic group \cite{GM08,GrovesManningSisto_RelativeCannon}.
We focus on Bowditch's model in this paper.

The definition of the cusped space involves the warped product of a pair of spaces, defined as follows.

\begin{defn}[Warped product]
Let $(F,d_F)$ and $(B,d_B)$ be two length spaces, and let $f \colon B \to [0,\infty)$ be a continuous function called the \emph{warping function}.
Let $\sigma \colon [0,1] \to F \times B$ be a path in $F\times B$ with component functions $\gamma\colon [0,1]\to F$ and $\beta\colon [0,1] \to B$.  The \emph{length} of $\sigma$ is given by
\[
   \ell(\sigma) = \sup_\tau \sum_{i=1}^{n(\tau)}  \Bigl( f^2 \bigl( \beta(t_i) \bigr) \, d^2_F \bigl( \gamma(t_i), \gamma(t_{i-1}) \bigr) + d^2_B \bigl( \beta(t_i),\beta(t_{i-1}) \bigr)  \Bigr)^{1/2}
\]
where the supremum is taken over all partitions $\tau$ of $[0,1]$ given by $0=t_0<t_1<\cdots<t_{n(\tau)}=1$.
If the warping function $f$ has no zeros then the length $\ell(\sigma)$ defined above determines a length metric $d_f$ on the set $F \times B$.  The resulting metric space is the \emph{warped product} $F \times_f B$
\end{defn}

\begin{defn}[The cusped space]
\label{def: cusped space}
Suppose $\Gamma$ is any metric graph.  The \emph{metric horoball} based on $\Gamma$ is the warped product
\[
   \text{cusp}(\Gamma) = \Gamma  \times_{e^{-t}} [0,\infty)
\]
Let $G$ be a group with a generating set $A$ and a collection $\mathbb{P}$ of finitely many finitely generated subgroups called \emph{peripheral subgroups}.
The generating set $A$ is \emph{adapted to $\mathbb{P}$} if for each $P \in \mathbb{P}$ the set $A \cap P$ generates $P$.
Let $\Gamma=\Gamma(G,A)$ be the Cayley graph of $G$ with respect to a finite generating set $A$ that is adapted to $\mathbb{P}$.  We endow $\Gamma$ with a length metric in which each edge has length one.
A \emph{peripheral coset} in $(G,\mathbb{P})$ is a coset $gP$ with $g \in G$ and $P \in \mathbb{P}$. For each peripheral coset $gP$, let $\Gamma_{gP}$ be the full subgraph of $\Gamma$ with vertex set $gP$.  Note that $\Gamma_{gP}$ is a copy of the Cayley graph of $P$ with respect to the generating set $A \cap P$.

The \emph{cusped space} $X(G,\mathbb{P},A)$ is the metric $2$--complex formed from the Cayley graph $\Gamma$ by attaching a metric horoball $\text{cusp}(\Gamma_{gP})$ to $\Gamma_{gP}$ for each peripheral coset $gP$.
In this attaching, we identify the copy of $\Gamma_{gP}$ in $\Gamma$ with the isometric copy $\Gamma_{gP} \times \{0\}$ in $\text{cusp}(\Gamma_{gP})$. 
A point of $X(G,\mathbb{P},A)$ is either a point of the Cayley graph $\Gamma$ or a triple $(gP,x,t)$ where $gP$ is a peripheral coset, $x$ is a point of $\Gamma_{gP}$, and $t \in [0,\infty)$.
\end{defn}
 

\begin{defn}[Relative hyperbolicity]
\label{def: rel hyp}
Suppose $G$ is a group, $\mathbb{P}$ is a finite family of subgroups of $G$, and $A$ is a finite generating set for $G$ that is adapted to $\mathbb{P}$.  Then $(G,\mathbb{P})$ is \emph{relatively hyperbolic} if the cusped space $X(G,\mathbb{P},A)$ is $\delta$--hyperbolic for some $\delta>0$. The \emph{Bowditch boundary} $\boundary(G,\mathbb{P})$ is the Gromov boundary of the cusped space $X$.
\end{defn}

By \cite{GM08,GrovesManningSisto_RelativeCannon} the definition is equivalent to other definitions of relative hyperbolicity in the literature. 
We note that relative hyperbolicity does not depend on the choice of finite adapted generating set (see \cite{GrovesManningSisto_RelativeCannon}).  If $A$ and $A'$ are two finite adapted generating sets, there is a $G$--equivariant homeomorphism between the boundaries of the respective cusped spaces by \cite{BowditchRelHyp}.

\begin{theorem}[Bowditch, Dasgupta]
\label{thm:BoundaryConnectedness}
Let $(G,\mathbb{P})$ be relatively hyperbolic.
The boundary $M=\boundary(G,\mathbb{P})$ is a compact, metrizable space with the following connectedness properties:
\begin{enumerate}
 \item
 \label{item:Connected}
 The boundary is connected if and only if every $P \in \mathbb{P}$ is infinite and $G$ does not split over a finite subgroup relative to $\mathbb{P}$.
 \item
 \label{item:LocallyConnected}
 Suppose $M$ is connected. Then $M$ is locally connected, and every global cut point of $M$ is a parabolic fixed point, \emph{i.e.}, the unique fixed point of a conjugate of a peripheral subgroup.
\end{enumerate}
\end{theorem}

\begin{proof}
It is well known that a proper $\delta$--hyperbolic space always has compact, metrizable boundary \cite{Gro87}.
If some peripheral subgroup is finite, then the boundary has isolated points.
When all peripheral subgroups are infinite, the equivalence in (\ref{item:Connected}) is due to Bowditch \cite{BowditchRelHyp}.
Conclusion (\ref{item:LocallyConnected}) is a theorem of Dasgupta, which was established earlier in a special case by Bowditch \cite{Dasgupta_Thesis,Bow01_Peripheral}.
(For the original hyperbolic case, see \cite{Swa}.)
\end{proof}

If $(G,\P)$ is relatively hyperbolic, its relatively quasiconvex subgroups are a natural family of subgroups that plays a key role in organizing the study of relatively hyperbolic groups (see for example \cite{AgolGrovesManning_QCERF,Hru10,Agol13,HruskaWise_Finiteness,WiseHierarchies}).

We now introduce two equivalent notions of relative quasiconvexity that will be used in Section~\ref{sec: trees separate}. For finitely generated subgroups, these definitions are known to be equivalent by a result of Manning--Mart\'{i}nez \cite{ManningMartinez10}. 

\begin{defn}[QC-1]
\label{def: AGM RC}
Let $(G,\bbp)$ and $(H,\mathbb{O})$ be relatively hyperbolic, and let $A$ and $B$ be finite generating sets of $G$ and $H$ that are adapted to $\bbp$ and $\mathbb{O}$ respectively.  Let $\phi\colon H \rightarrow G$ be a monomorphism. Assume that $\phi(O)$ is conjugate in $G$ into some $P\in\bbp$ for each $O\in\mathbb{O}$. We extend $\phi$ to a map $\check{\phi}\colon X(H,\mathbb{O},B)\hookrightarrow X(G,\bbp,A)$ as follows. For each $O_i\in\mathbb{O}$ choose an element $g_i\in G$ (of minimal length) and some $P_{i} \in \mathbb{P}$ such that $\phi(O_i)\subset g_iP_{i}g_i^{-1}$. For each $h\in H$ define $\check{\phi}(h)=h$, and extend equivariantly to a map from the Cayley graph of $H$ to the Cayley graph of $G$.
For each point of a horoball of $X(H,\mathbb{O},B)$ define
\[
   \check{\phi}(kO_i,x,t)=(\phi(k)g_iP_{i},\phi(x)g_i,t).
\]
A subgroup $H$ of $G$ is {\it relatively quasiconvex} if $H$ has a peripheral structure such that the extension $\check{\iota}$ of the inclusion map $\iota\colon H\hookrightarrow G$ has quasiconvex image.
\end{defn}

\begin{defn}[QC-2]
\label{def: Hruska rel qc}
Assume that $(G, \bbp)$ is relatively hyperbolic with cusped space $X(G,\mathbb{P},A)$.
A subgroup $H$ of $G$ is {\it relatively quasiconvex} if for any choice of base point $x$ in the Cayley graph $\Gamma(G,A)$ there exists a constant $\mu$ such for any geodesic $c$ in $X$ connecting two points of $Hx$ the set $c\cap \Gamma(G,A)$ is contained in the $\mu$--neighborhood of $Hx$ with respect to the word metric. 
\end{defn}

\begin{theorem}[\cite{ManningMartinez10}, Thm.~A.10]
\label{thm: ManningMartinez}
Let $(G,\bbp)$ be relatively hyperbolic with finite adapted generating set $A$. Let $H$ be a finitely generated subgroup of $G$. Then $H$ satisfies \textup{(QC-1)} if and only if $H$ satisfies \textup{(QC-2)}.
\end{theorem}


The \emph{limit set} $\Lambda H$ of a subgroup $H \le G$ is defined to be the set of limit points of any $H$--orbit in the Bowditch boundary $\boundary(G,\bbp)$.

The following result on intersections of relatively quasiconvex subgroups holds for all countable $G$ and all countable subgroups $H$ and $K$, regardless of finite generation.
In order to be consistent with the finitely generated definitions given above, slight restrictions would be required.  In this article we only require the finitely generated cases of the proposition.

\begin{proposition}
\label{prop:IntersectionQC}
Let $(G,\bbp)$ be relatively hyperbolic.
Suppose $H$ and $K$ are relatively quasiconvex subgroups.
Then $H \cap K$ is also relatively quasiconvex.
Furthermore the intersection $\Lambda H \cap \Lambda K$ is equal to the union of $\Lambda(H\cap K)$ together with any parabolic points contained in both $\Lambda H$ and $\Lambda K$.
\end{proposition}

The fact that the intersection is relatively quasiconvex is proved in various levels of generality by Dahmani, Osin, Mart\'{i}nez, and the second author \cite{Dahm03,Osin06,Martinez09_Combination,Hru10}.
Dahmani proves the conclusion regarding intersections of limit sets in \cite[Prop.~1.10]{Dahm03} in the setting of fully quasiconvex subgroups.  Wenyuan Yang \cite{Yang_LimitSets} establishes the general case.

\section{Bowditch boundaries with no cut points}
\label{sec:NoCutPoint}

Bowditch's construction of the JSJ tree for a one-ended hyperbolic group $G$ over the family of $2$--ended subgroups in \cite{Bow98_JSJ} depends heavily on a classification of the local cut points of the Gromov boundary $\boundary G$.  Bowditch proves that the multivalent points are partitioned into disjoint multivalent exact cut pairs, each bivalent point is contained in either an isolated exact cut pair or a necklace, and $\boundary G$ cannot contain apeirovalent points.

As suggested by Bowditch \cite{Bow01_Peripheral} many results in \cite{Bow98_JSJ} have natural extensions to relatively hyperbolic groups whose boundary is connected and without global cut points, provided that one modifies them to take into account the existence of parabolic points in the boundary.
Recall that by Theorem~\ref{thm:BoundaryConnectedness}, the boundary of such a group is a Peano continuum without cut points.
Guralnik and the first author \cite{Gur,Haulmark_LocalCutPt} have generalized some results of \cite{Bow98_JSJ} in this manner.
The main goal of this section is to advance this generalization by proving Proposition~\ref{prop:Bivalent}, which classifies the bivalent points of the Bowditch boundary, and Proposition~\ref{prop: only finitely many}, which establishes that the inseparable exact cut pairs lie in only finitely many orbits.

Throughout this section, $(G,\mathbb{P})$ denotes a relatively hyperbolic group with boundary $M=\boundary(G,\mathbb{P})$ such that $M$ is connected and has no global cut point.

Any nonparabolic $2$--ended subgroup of $G$ fixes a unique pair of points $\{\zeta,\xi\}$ in $M$, and acts cocompactly on $M \setminus \{\zeta,\xi\}$ (see, for example \cite[Thm.~2I]{Tuk94}).
Using this fact, the next lemma follows by the same proof as \cite[Lem.~5.6]{Bow98_JSJ}.

\begin{lemma}[Loxodromic $\Longrightarrow$ exact]
\label{lem: 2ended are exact}
Assume $M=\boundary(G,\mathbb{P})$ is connected and has no cut points. If a cut pair $\{\zeta,\xi\}$ is the limit set of a nonparabolic $2$--ended subgroup, then it is exact. \qed
\end{lemma}

Combining a result of Guralnik \cite[Prop.~4.7]{Gur} with Lemma~\ref{lemma: cut not separated by M(3+)}, gives the following result.

\begin{proposition}[Multivalent $\Longrightarrow$ loxodromic]
\label{prop:M3Pairs}
Assume $M=\relbndry$ is connected with no cut points.  The multivalent points of $M$ are partitioned into disjoint inseparable exact cut pairs.  Each such pair is the limit set of a nonparabolic $2$--ended group.
In particular, the set of multivalent points of $M$ is countable. \qedhere
\end{proposition}

The next result classifying the apeirovalent points of $M$ was observed by Guralnik \cite[Prop.4.2]{Gur}.

\begin{proposition}[Apeirovalent $\Longrightarrow$ parabolic]
\label{prop:ApeirovalentParabolic}
Suppose $M=\relbndry$ is connected with no cut points.  Then every apeirovalent point of $M$ is parabolic.  In particular, the set of apeirovalent points of $M$ is countable.
\end{proposition}

In the following three results we study the bivalent points of $\bndry(G,\bbp)$. 

\begin{lemma}
\label{lem:parabolics are not isolated}
Suppose $M=\relbndry$ is connected with no cut points.
Let $N$ be a $\sim$--class of $M$ containing more than one point.
If $\eta \in N$ is parabolic then $N$ is infinite and $\eta$ is an accumulation point of $N$.
In particular, for each point $\xi\in N \setminus\{\eta\}$, the orbit of $\xi$ under the stabilizer of $N$ is infinite and has $\eta$ as a limit point.
\end{lemma}

\begin{proof}
Suppose $\eta\in N$ is parabolic. Let $\xi\in N\setminus\{\eta\}$. By \cite{BowditchRelHyp} the infinite peripheral subgroup $P=\stab(\eta)$ acts properly and cocompactly on $M\setminus\{\eta\}$. Thus $\eta$ is a limit point of the orbit $P(\xi)$ in $M$.  Since $P$ acts on $M$ by homeomorphisms, the hypothesis $\xi \sim \eta$ implies $p(\xi) \sim \eta$ for every $p\in P$. In particular, $P(\xi)$ is an infinite subset of $N$ accumulating at $\eta$.
\end{proof}

In \cite{Haulmark_LocalCutPt}, a necklace is defined to be the closure of a set formed by intersecting a $\sim$--class with the set of all conical limit points.
The following lemma implies that the purely topological definition of necklace in Section~\ref{subsec:CutPoints} coincides with the dynamical definition of necklace given in \cite{Haulmark_LocalCutPt}.

\begin{lemma}
\label{lem: M vs M*}
Assume that $M=\relbndry$ is connected and without cut points. Every $\sim$--class $N$ with at least three points contains a dense set of conical limit points.
In particular the closure of $N$ is equal to the closure of its set of conical limit points.
\end{lemma}

\begin{proof}
Suppose there exists a $\sim$--class $N$ containing only parabolic points.  By \cite[Lem.~3.7]{Bow98_JSJ} for any necklace $\overline{N}$ in a Peano continuum with no cut points,
the set $\overline{N}\setminus N$ contains only multivalent and apeirovalent points. 
By Propositions \ref{prop:M3Pairs} and~\ref{prop:ApeirovalentParabolic} there are only countably many such points in $M$.  Since $N$ is also countable, it follows that $\overline{N}$ is countable.
On the other hand, Lemma~\ref{lem:parabolics are not isolated} implies that $\overline{N}$ is a nonempty, compact metric space with no isolated points.  In particular, $\overline{N}$ has the cardinality of the continuum (see \cite[Cor.~6.3]{Kechris95}),
a contradiction.

By the previous paragraph, a $\sim$--class $N$ must contain at least one conical limit point.  It follows from Lemma~\ref{lem:parabolics are not isolated} that the conical limit points of $N$ are dense in $N$.
\end{proof}

The first author classifies the bivalent conical limit points in \cite{Haulmark_LocalCutPt}.  In the following proposition, we extend this classification to cover all possible bivalent points. This proposition also extends a related result of Guralnik \cite[Prop.~4.8]{Gur} by giving more detailed information about necklaces.

\begin{proposition}[Bivalent points]
\label{prop:Bivalent}
Suppose $M=\relbndry$ is connected with no cut points and not homeomorphic to a circle.
If $\xi\in M$ is bivalent, then exactly one of the following holds:
   \begin{enumerate}
   \item $[\xi]_\sim$ is a single parabolic point.
   \item $[\xi]_\sim$ is an inseparable cut pair and is the limit set of a nonparabolic $2$--ended subgroup of $G$.
   \item
   \label{item:Bivalent:Necklace}
   The closure $\overline{N}$ of $N=[\xi]_\sim$ is a cyclically separating Cantor set of $M$, the stabilizer $H$ of $\overline{N}$ is a relatively quasiconvex subgroup with limit set $\overline{N}$, and each jump of $\overline{N}$ is an inseparable exact cut pair stabilized by a nonparabolic $2$--ended subgroup of $H$.
   Furthermore each necklace $\overline{N}$ contains only finitely many $\Stab(\overline{N})$--orbits of jumps.
   \end{enumerate}
\end{proposition}

\begin{proof}
Suppose first that $\xi$ is a bivalent point and $[\xi]_\sim$ contains only one point.
By \cite[Thm.~4.20]{Haulmark_LocalCutPt} every bivalent conical limit point is involved in an exact cut pair, so $\xi$ must be parabolic.

Now suppose $[\xi]_\sim$ contains exactly two points.  By Lemma~\ref{lem:parabolics are not isolated} a finite $\sim$--class containing a parabolic point must be a singleton, so the two points of $[\xi]_\sim$ must be conical limit points.  However, as observed in \cite[Thm.~4.20]{Haulmark_LocalCutPt} every $\sim$--class containing exactly two conical limit points must be the limit set of a nonparabolic $2$--ended subgroup of $G$.

If $[\xi]_\sim$ contains at least three elements, then its closure is a necklace $\overline{N}$.
As mentioned above, \cite{Haulmark_LocalCutPt} uses a slightly different definition of necklace than used here, but Lemma~\ref{lem: M vs M*} establishes that the two definitions are equivalent.  Conclusion~(\ref{item:Bivalent:Necklace}) now follows from the proof of \cite[Lem.~4.19]{Haulmark_LocalCutPt} together with Lemma~\ref{lem: 2ended are exact}.
\end{proof}

The following proposition classifies all possible inseparable exact cut pairs in a connected Bowditch boundary with no cut points.

\begin{proposition}[Inseparable exact cut pairs]
\label{prop: only finitely many}
Suppose $M=\boundary(G,\mathbb{P})$ is connected and has no cut points. Then:
   \begin{enumerate}
   \item $M$ contains only finitely many $G$--orbits of inseparable exact cut pairs.
   \item Each inseparable exact cut pair has a non-parabolic $2$--ended stabilizer and is the limit set of that $2$--ended group.
   \item If two inseparable exact cut pairs intersect then they must coincide.
   \end{enumerate}
\end{proposition}

\begin{proof}
Since a circle contains no inseparable cut pairs, we assume that $M$ is not homeomorphic to a circle. 
By definition every inseparable exact cut pair is either isolated or in a $\sim$--class whose closure is a necklace.
By a theorem of Gerasimov \cite{Ger09}, the geometrically finite convergence group $G$ acts cocompactly on the space of ordered pairs of distinct points of $M$.
Let $\Delta$ denote the diagonal of $M\times M$. Observe that the space of distinct pairs $(M\times M) \setminus \Delta$ has an exhaustion by compact sets $C_\epsilon = \bigset{ (\zeta,\xi) }{ d(\zeta,\xi) \ge \epsilon }$ for all $\epsilon>0$.
Expressing Gerasimov's result in terms of this exhaustion, we see that, there exists $\epsilon>0$ such that any subset $S \subset M$ with $\abs{S}=2$ has a $G$--translate with diameter at least $\epsilon$.

By Lemmas 3.15 and~3.16 of \cite{Bow98_JSJ}, only finitely many isolated exact cut pairs and finitely many necklaces in $M$ have diameter greater than $\epsilon$ (for any positive $\epsilon$).
Therefore the exact cut pairs of $M$ lie in only finitely many $G$--orbits of isolated exact cut pairs and necklaces.
But each of these finitely many necklaces $\overline{N}$ contains only finitely many $\Stab(\overline{N})$--orbits of jumps by Proposition~\ref{prop:Bivalent}. Therefore $M$ contains only finitely many $G$--orbits of inseparable exact cut pairs that are jumps of necklaces. In particular the exact cut pairs of $M$ lie in finitely many $G$--orbits.

By definition, any inseparable cut pair contained in a necklace must be a jump.
By Propositions \ref{prop:M3Pairs} and~\ref{prop:Bivalent}, each inseparable exact cut pair is the fixed point set of a loxodromic element.
For any two loxodromic elements, their fixed point sets are either disjoint or equal (see \cite{Tuk94}), completing the proof of the lemma.
\end{proof}

\section{Splittings over relatively quasiconvex subgroups}
\label{sec: trees separate}

The main results in this paper depend on the following splitting theorem due, independently, to Bigdely--Wise and Guirardel--Levitt.

\begin{proposition}[\cite{BigdelyWise13},\cite{GuirardelLevitt15_AutRelHyp}]
\label{prop:BigdelyWise}
Let $(G,\bbp)$ be relatively hyperbolic. Suppose $G$ acts on a tree $T$ relative to peripheral subgroups.
If all edge stabilizers are finitely generated and relatively quasiconvex, then all vertex stabilizers are finitely generated and relatively quasiconvex.
\end{proposition}

In this section, we prove a more general result, in which we drop the assumption that the action is relative to peripheral subgroups.
Although this stronger result is not needed in the rest of this paper, it has been used in Wise's study of toral relatively hyperbolic groups with quasiconvex hierarchies \cite{WiseHierarchies}.

\begin{proposition}
\label{prop: rel qc}
Let $(G,\bbp)$ be relatively hyperbolic. If $G$ acts on a tree $T$ with finitely generated, relatively quasiconvex edge stabilizers, then the stabilizer of any vertex of $T$ is also finitely generated and relatively quasiconvex.  
\end{proposition}

Proposition~\ref{prop: rel qc} is claimed in \cite[Lem.~4.9]{BigdelyWise13}.
Unfortunately the proof provided there is not enough to establish Proposition~\ref{prop: rel qc}, but instead establishes only the weaker result of Proposition~\ref{prop:BigdelyWise}.

In personal communication, Wise explained to the authors an argument for why Proposition~\ref{prop: rel qc} should be true using the criterion for relative quasiconvexity in \cite{MartinezWise11}.
In the proof presented below, we instead rely on the strong characterization of relative quasiconvexity due to Agol--Groves--Manning (see Definition~\ref{def: AGM RC}) as well as the weaker characterization due to the second author (see Definition~\ref{def: Hruska rel qc}).
The argument presented here is somewhat different from Wise's suggestion in the details, but close in spirit.

\begin{proof}[Proof of Proposition~\ref{prop: rel qc}]
Let $(G,\bbp)$ be relatively hyperbolic, and assume $T$ is a simplicial tree on which $G$ acts with finitely generated, relatively quasiconvex edge stabilizers.
Without loss of generality, we may assume that the action is minimal and without inversions.
Let $\Gamma(G,A)$ denote the Cayley graph of $G$ with respect to a finite adapted generating set $A$, and let $X=X(G,\mathbb{P},A)$ be the corresponding cusped space.
We construct a map $\pi\colon X\rightarrow T$  as a composition of maps $\alpha\circ\phi$ where $\phi\colon X\rightarrow \Gamma(G,A)$ is given by projecting each horoball $\text{cusp}(gP_i)$ to its base coset $gP_i$ in $\Gamma(G,A)$, and $\alpha$ is given by an orbit map. Namely,  $\alpha\colon\Gamma(G,A)\rightarrow T$ is obtained by arbitrarily choosing a point in $T$ as the image of the identity vertex of $\Gamma(G, A)$ then mapping all other vertices of $\Gamma(G,A)$ to $T$ equivariantly. We extend the map to send the edges of $\Gamma(G, A)$ to geodesics in $T$ equivariantly.  We note that the map $\alpha$ is continuous with respect to the CW topology on $T$, and hence also with respect to the coarser metric topology.

Let $v$ be a vertex of $T$.  Define $S_v$ to be the union of $v$ and all adjacent half-edges, where by half-edge we mean the edge segment connecting $v$ to the midpoint $m_e$ of the edge adjacent $e$ to $v$. 

In order to verify that $G_v$ satisfies Definition~\ref{def: Hruska rel qc}, we consider an arbitrary geodesic $\sigma$ in $X$ with endpoints in $G_v$. 
We will show that the intersection $\sigma \cap \Gamma(G,A)$ lies within a uniformly bounded neighborhood of $G_v$.

In the tree, the set $S_v$ is bounded by the midpoints $m_e$ of neighboring edges in the following sense: any path with endpoints in $S_v$ whose interior is disjoint from $S_v$ must be a loop with both endpoints at $m_e$ for some edge $e$ adjacent to $v$.

By continuity of $\alpha$, we get a corresponding separation property in $\Gamma(G,A)$ as follows:
any path with endpoints in $\alpha^{-1}(S_v)$ whose interior is disjoint from $\alpha^{-1}(S_v)$ has both endpoints in the same set $\alpha^{-1}(m_e)$, where $e$ is an edge adjacent to $v$.
Thus the map $\pi=\alpha \of \phi \colon X \to T$ has a similar separation property.

Note that any minimal action of a finitely generated group $G$ on a tree $T$ has only finitely many orbits of vertices and edges \cite[Cor.~3.8]{SW79}.  Let $\mathcal{V}$ and $\mathcal{E}$ be sets of representatives of the finitely many $G$--orbits in $T$ of vertices and edges respectively.
For each $e \in \mathcal{E}$, the set $\alpha^{-1}(m_{e})$ is within a finite Hausdorff distance of the subgroup $G_e = \Stab(e)$ (see for instance \cite[Lem.~7.3]{HR1}). 
Similarly for each $v \in \mathcal{V}$, the set $\alpha^{-1}(S_v)$ is within finite Hausdorff distance of $G_v=\Stab(v)$.
For the other edges $xe$ with $x \in G$ and $e \in \mathcal{E}$, the set $\alpha^{-1}(m_{xe}) = x \alpha^{-1}(m_e)$ is within a uniformly bounded Hausdorff distance of the coset $x G_e$.
A similar conclusion holds as well for the vertices of an orbit.

Informally, for each edge $e \in \mathcal{E}$ the set $\pi^{-1}(m_e)$ is coarsely equal to a copy of $G_e$ together with cusp regions corresponding to the intersections of $\alpha^{-1}(m_e)$ with coset graphs $\Gamma_{gP}$.
Since $G_e$ is relatively quasiconvex, naively one might expect the sets $\pi^{-1}(m_e)$ to be quasiconvex as in Definition~\ref{def: AGM RC}.
However in general they are not quasiconvex, due to the presence of accidental parabolics.

An \emph{accidental parabolic subgroup} of $G_e$ is an infinite subgroup of the form $G_e \cap gPg^{-1}$ such that $\alpha^{-1}(m_e)$ does not intersect the subgraph $\Gamma_{gP}$.  Such a subgroup is analogous to an accidental parabolic of a geometrically finite (\emph{i.e.}, relatively quasiconvex) surface in a finite volume hyperbolic $3$--manifold. A geometrically finite surface in $\Hyp^3$ that contains accidental parabolics is not a quasiconvex subset of $\Hyp^3$ because certain cusp regions are missing from the surface but are contained in its convex hull.

For each edge $e \in \mathcal{E}$, let $\mathcal{M}_e$ be the collection of cosets $gP$ such that $g\in G$, $P\in\mathbb{P}$, and  $\alpha^{-1}(m_e)\cap \Gamma_{gP}$ is a non-empty, bounded set.  
Since $G_e$ acts cocompactly on $\alpha^{-1}(m_e)$ and the family of peripheral cosets $\set{gP}{g\in G \text{ and } P \in \mathbb{P}}$ is locally finite, the members of $\mathcal{M}_e$ lie in finitely many $G_e$--orbits.  Choose representatives $g_1 P_1,\dots,g_k P_k$ for the $G_e$--orbits and consider the corresponding finite subgroups
\[
   G_e \cap g_1 P_1 g_1^{-1}, \ldots,
   G_e \cap g_k P_k g_k^{-1}.
\]
We note that a finite subgroup could occur multiple times in the above list.  However the indexing of the list keeps track of the distinct peripheral cosets associated with the finite subgroups.
Since $G_e$ is relatively quasiconvex, the cosets $gP$ such that $G_e \cap gPg^{-1}$ is infinite are also in finitely many $G_e$--orbits 
(see \cite[Lem.~6.7]{HruskaWise_Finiteness}) with representatives $g_{k+1} P_{k+1},\dots,g_\ell P_\ell$ and corresponding subgroups
\[
   G_e \cap g_{k+1} P_{k+1} g_{k+1}^{-1}, \ldots,
   G_e \cap g_\ell P_\ell g_\ell^{-1}.
\]
This list includes all cosets $g_iP_i$ for which the intersection $\alpha^{-1}(m_e) \cap \Gamma_{g_i P_i}$ is unbounded, as well as possibly some cosets such that $\Gamma_{g_iP_i}$ does not intersect $\alpha^{-1}(m_e)$ at all.
The latter type correspond to the accidental parabolic subgroups of $G_e$.

Let $\mathbb{O}_e$ be the indexed family of subgroups $O_1,\dots,O_\ell$, where $O_i = G_e \cap g_iP_ig_i^{-1}$.
Then $(G_e,\mathbb{O}_e)$ is relatively hyperbolic.
Indeed by relative quasiconvexity, $G_e$ is hyperbolic relative to the infinite subgroups $O_{k+1},\dots,O_\ell$ (see \cite{Hru10}).
Since finitely many finite subgroups can be added to the peripheral structure of any relatively hyperbolic group \cite{Osin06Elementary}, it follows that $G_e$ is also hyperbolic relative to $\mathbb{O}_e$.

Let $\iota\colon G_e \hookrightarrow G$ be the inclusion, and consider for each $O_i$ the canonical inclusion $O_i = G_e \cap g_iP_ig_i^{-1} \hookrightarrow g_iP_ig_i^{-1}$.
Since $G_e$ is relatively quasiconvex, Definition~\ref{def: AGM RC} implies that $\check{\iota}\big(X(G_e,\mathbb{O}_e)\big)$ is a $\kappa$--quasiconvex subspace of $X(G,\bbp)$ for some constant $\kappa$ not depending on the choice of $e \in \mathcal{E}$. The group $G_e$ acts cocompactly on $\alpha^{-1}(m_e)$ and there are only finitely many orbits of edges in $\mathcal{E}$, so there exists a  constant $\rho<\infty$ such that $\alpha^{-1}(m_e)$ is contained in the $\rho$--neighborhood of $G_e$.
By choice of $\mathbb{O}_e$, if $gP$ is any peripheral coset in $G$ such that $\alpha^{-1}(m_e) \cap \Gamma_{gP}$ is nonempty, then $gP = h g_i P_i$ for some $i$ and some $h \in G_e$.
Translating by the isometry $h^{-1}$, we see that $\alpha^{-1}(m_e) \cap \Gamma_{g_iP_i}$ is nonempty and lies within a $\rho$--neighborhood of $G_e$ and within a $\abs{g_i}$--neighborhood of $g_i P_i g_i^{-1}$. Therefore it lies within an $L$--neighborhood of the intersection $G_e \cap g_i P_i g_i^{-1} = O_i$ for some constant $L$ depending only on $\rho$ and the finitely many lengths $\abs{g_j}$ for $j=1,\dots,\ell$ by \cite[Prop.~9.4]{Hru10}.
Translating back by $h$, we conclude that 
$\alpha^{-1}(m_e) \cap \Gamma_{gP}$ lies within an $L$--neighborhood of $hO_i$.
Thus $\pi^{-1}(m_e)$ is contained in the $\rho$--neighborhood of $\check{\iota}\big(X(G_e,\mathbb{O}_e)\big)$. Note that action of $G$ on $X(G,\bbp)$ is by isometries, so for every $x\in G$ and $e \in \mathcal{E}$ the set $\pi^{-1}(m_{xe})$ is contained in the $\rho$--neighborhood of the quasiconvex set $x\check{\iota}\big(X(G_e,\mathbb{O}_e)\big)$.

For each $v \in \mathcal{V}$, we will show that $G_v= \Stab(v)$ is relatively quasiconvex by verifying Definition~\ref{def: Hruska rel qc}.
Observe that the set $\alpha^{-1}(S_v)$ is within a finite Hausdorff distance of the subgroup $G_v = \Stab(v)$.
We will show that any geodesic $\sigma$ with endpoints in $\alpha^{-1}(S_v)$ lies in a bounded neighborhood of the union of $G_v$ and the family of all horoballs of $X$.

The geodesic $\sigma$ is a concatenation of finitely many arcs that lie in $\pi^{-1}(S_v)$ and finitely many arcs with endpoints in $\pi^{-1}(S_v)$ whose interior is disjoint from $\pi^{-1}(S_v)$.
Since $\pi^{-1}(S_v)$ lies in a bounded neighborhood of the union of $G_v$ and the family of all horoballs of $X$, it suffices to consider geodesic segments $\sigma'$ of the second type: those with endpoints in $\pi^{-1}(S_v)$ whose interior is disjoint from $\pi^{-1}(S_v)$.

Any such geodesic $\sigma'$ must lie in a uniformly bounded neighborhood of the $\kappa$--quasiconvex set $x\check{\iota}\big(X(G_e,\mathbb{O}_e)\big)$ for some edge $xe$ incident to $v$ in the tree $T$.
As before the edges adjacent to $v$ lie in finitely many $G_v$--orbits represented by edges $x_1 e_1,\dots,x_r e_r$ with $x_j \in G$ and $e_j \in \mathcal{E}$.
Since this collection is finite, there is an upper bound on the Hausdorff distance between the coset $x_j G_{e_j}$ and the subgroup $x_j G_{e_j} x_j^{-1} \le G_v$.  Thus the cosets $x_j G_{e_j}$ for $j=1,\dots,r$ lie in a uniformly bounded neighborhood of $G_v$.
It follows that if $x\in G$ and $e\in \mathcal{E}$ and the edge $xe$ is incident to $v$, then the coset $x G_e$ lies in a uniformly bounded neighborhood of $G_v$.

Therefore the intersection $\sigma' \cap \Gamma(G,A)$ must lie within a uniformly bounded neighborhood of $\alpha^{-1}(S_v)$, and hence also of $G_v$. Thus $G_v$ satisfies the criterion for relative quasiconvexity given in Definition~\ref{def: Hruska rel qc}.
\end{proof}

\section{Elementary splittings and decompositions of the boundary}
\label{sec:Decompositions}

Suppose $(G,\bbp)$ is relatively hyperbolic. Let $\mathbb{E}$ be the collection of infinite elementary subgroups of $(G,\bbp)$ and let $\mathbb{E}_1\subseteq\mathbb{E}$ consist of the finitely generated, infinite elementary subgroups.
The main result of this section is Proposition~\ref{Prop: separation}, which states that any $(\mathbb{E}_1,\mathbb{P})$--tree $T$ on which $G$ acts is dual to a family of separations of the boundary corresponding to the edges of $T$.
The special case in which $G$ is word hyperbolic and $\mathbb{P}$ is empty is due to Bowditch \cite[\S 1]{Bow98_JSJ}.
The restriction to $(\mathbb{E}_1,\mathbb{P})$--trees does not affect the generality of our main results, since every $(\mathbb{E},\mathbb{P})$--tree is dominated by an $(\mathbb{E}_1,\mathbb{P})$--tree
\cite[Cor.~2.25]{GL_JSJ}.
We remark that the results of this section do not use the local connectivity of the boundary.


We note that each vertex group of an $(\mathbb{E}_1,\mathbb{P})$--tree $T$ is relatively hyperbolic in a natural way.

\begin{lemma}
Suppose $(G,\mathbb{P})$ is relatively hyperbolic, and suppose $T$ is an $(\mathbb{E}_1,\mathbb{P})$--tree.
Then each vertex stabilizer $G_v$ is relatively quasiconvex.  In particular, $(G_v,\mathbb{O}_v)$ is relatively hyperbolic, where $\mathbb{O}_v$ denotes a set of representatives of the conjugacy classes of infinite subgroups $G_v \cap gPg^{-1}$ for $g \in G$ and $P \in \mathbb{P}$.
Furthermore, the limit set of $G_v$ in $\boundary(G,\mathbb{P})$ is equivariantly homeomorphic to the Bowditch boundary $\boundary(G_v,\mathbb{O}_v)$.
\end{lemma}

\begin{proof}
Since each edge group is elementary, it is relatively quasiconvex.
By Proposition~\ref{prop:BigdelyWise}, each vertex group is also relatively quasiconvex.
Such a subgroup is relatively hyperbolic with Bowditch boundary equal to its limit set by \cite{Dahm03,Hru10}.
\end{proof}

We review the notion of the tree of cylinders associated to a splitting, a notion that is implicit in work of Bowditch \cite{Bow98_JSJ,Bow01_Peripheral} and was explicitly introduced by Guirardel--Levitt \cite{GL_TreesCyl}.

\begin{defn}[Tree of cylinders]
Let $(G,\mathbb{P})$ be relatively hyperbolic.
Given an $(\mathbb{E},\mathbb{P})$--tree $T$, edges $e$ and $e'$ are \emph{equivalent}, denoted $e \sim e'$, if their stabilizers are \emph{co-elementary}, \emph{i.e.,} if the subgroup $\langle G_e, G_{e'} \rangle$ generated by $G_e$ and $G_{e'}$ is elementary. For each equivalence class, the smallest subtree of $T$ containing the equivalence class is a \emph{cylinder}.
The \emph{tree of cylinders} of $T$ is the bipartite graph $T_c$ with vertex set $\mathcal{V} \sqcup \mathcal{W}$, where $\mathcal{V}$ is the set of all cylinders of $T$ and $\mathcal{W}$ is the set of vertices of $T$ belonging to at least two distinct cylinders, such that vertices $v \in \mathcal{V}$ and $w \in \mathcal{W}$ are joined by an edge if $w \in v$. The graph $T_c$ is an $(\mathbb{E},\mathbb{P})$--tree by \cite[\S 4]{GL_TreesCyl}.
%
%
\end{defn}


\begin{theorem}[\cite{GL_TreesCyl}]
\label{thm:TreeOfCylinders}
Let $(G,\mathbb{P})$ be relatively hyperbolic.
Let $T$ be an $(\mathbb{E},\mathbb{P})$--tree.
Then the corresponding tree of cylinders $T_c$ is an acylindrical $(\mathbb{E},\mathbb{P})$--tree.
The trees $T$ and $T_c$ are equivalent in the sense of Section~\ref{sec:JSJ} and have the same family of nonelementary vertex stabilizers.
If two trees $T$ and $T'$ are equivalent, then they have the same tree of cylinders.
\end{theorem}

In \cite{Bow01_Peripheral}, Bowditch studies the special case of $(\mathbb{E},\mathbb{P})$--trees $T$ such that all edge stabilizers are finitely generated parabolic.  In this case, the tree of cylinders $T_c$ again has finitely generated parabolic edge stabilizers and is the minimal invariant subtree of a ``peripheral splitting'' in the sense of Bowditch.
The splitting $T_c$ is closely related to the cut point structure of the Bowditch boundary, as summarized in the following result proved in \cite[\S\S 7--8]{Bow01_Peripheral}.
(For related results in different settings, see \cite{Bow98_JSJ,Swiatkowski16_DenseAmalgam,HR1}.)

\begin{proposition}[A tree of spaces]
\label{prop:TreeOfSpaces}
Let $(G,\mathbb{P})$ be relatively hyperbolic with each $P \in \mathbb{P}$ infinite. Let $T$ be an $(\mathbb{E},\mathbb{P})$--tree with all edge groups finitely generated parabolic, and let $T_c$ be its tree of cylinders.
There exists an injective function $\phi \colon \mathcal{V}(T_c) \cup \boundary T_c \to M$ with the following properties.

\begin{enumerate}
\item For each $v \in \mathcal{V}(T_c)$ the point $\phi(v)$ is the unique fixed point of the parabolic subgroup $G_v$.
We refer to elements of $\phi(\boundary T_c)$ as \emph{ideal points}.
The limit sets $G_w$ for $w \in \mathcal{W}(T_c)$ are disjoint from the set of ideal points.
Two limit sets $G_w$ and $G_{w'}$ intersect, if at all, in an element of $\phi\bigl(\mathcal{V}(T_c) \bigr)$, and $v\in \mathcal{V}(T_c)$ is adjacent to $w \in \mathcal{W}(T_c)$ if and only if $\phi(v) \in \Lambda G_w$.

\item If $S$ is any subtree of $T_c$, the set $\Phi(S) = \phi\bigl( \boundary S\bigr) \cup \bigcup_{w \in \mathcal{W}(S)} \Lambda G_w$ is the closure in $M$ of the set $\bigcup_{w \in \mathcal{W}(S)} \Lambda G_w$.  Furthermore $\Phi(T_c)=M$.
%
%
\end{enumerate}
\end{proposition}

We now begin our study of the boundary of a relatively hyperbolic pair $(G,\mathbb{P})$ in relation to an $(\mathbb{E}_1,\mathbb{P})$--tree.

\begin{defn}[Directed edges]
A \emph{directed edge} $\vec{e}$ of a tree is an edge $e$ together with an orientation, \emph{i.e.}, a choice of one adjacent vertex as the \emph{initial} vertex $\iota(\vec{e}\,)$ and the other as the \emph{terminal} vertex $\tau(\vec{e}\,)$ of $\vec{e}$.  The same edge with opposite orientation is denoted $-\vec{e}$.
\end{defn}

\begin{defn}[Halfspaces]
Suppose $(G,\mathbb{P})$ is relatively hyperbolic with each $P \in \mathbb{P}$ infinite.  Let $T$ be an $(\mathbb{E}_1,\mathbb{P})$--tree.
Let $\vec{e}$ be any oriented edge of $T$.
Removing the open edge $e$ from $T$ leaves two components $T(\vec{e}\,)$ and $T(-\vec{e}\,)$, the first containing $\tau(e)$ and the second containing $\iota(e)$.
The \emph{closed halfspaces} of $\relbndry$ corresponding to $\vec{e}$ are the closed sets
\[
   H(\vec{e}\,) = \overline{\bigcup_{v \in V(T(\vec{e}\,))} \!\!\!\!\Lambda G_v}
   \qquad\text{and}\qquad
   H(-\vec{e}\,) = \overline{\bigcup_{v \in V(T(-\vec{e}\,))} \!\!\!\!\Lambda G_v}.
\]
The \emph{open halfspaces} corresponding to $e$ are $U(\vec{e}\,) = H(\vec{e}\,) \setminus \Lambda G_e$ and $U(-\vec{e}\,) = H(-\vec{e}\,) \setminus \Lambda G_e$.
\end{defn}

\begin{proposition}[\emph{cf.}\ \cite{Bow98_JSJ}, \S 1]
\label{Prop: separation}
Let $(G,\bbp)$ be relatively hyperbolic with each $P\in \mathbb{P}$ infinite. Let $T$ be a minimal $(\mathbb{E}_1,\bbp)$--tree.
For each oriented edge $\vec{e}$ of $T$, the open halfspaces $U(\vec{e}\,)$ and $U(-\vec{e}\,)$ are a separation of $\relbndry \setminus \Lambda G_e$ into two disjoint, nonempty open sets. \end{proposition}

\begin{proof}
Each nonparabolic $2$--ended edge group of $T$ stabilizes a pair of points.
Let $\mathbb{C}$ be a finite set of representatives of the conjugacy classes of $G$--stabilizers of such pairs of points.
Then each member of $\mathbb{C}$ is a nonparabolic maximal $2$--ended subgroup of $G$ (see \cite[Thm.~2I]{Tuk94}).  Let $\widehat{\mathbb{P}}$ equal the union $\mathbb{P} \cup \mathbb{C}$.
The group pair $(G,\widehat{\mathbb{P}})$ is relatively hyperbolic, and its Bowditch boundary $\boundary(G,\widehat{\mathbb{P}})$ is formed from $\boundary(G,\mathbb{P})$ by collapsing to a point each limit set of a nonparabolic edge group of $T$ by \cite{Dahm03} (see
\cite{Osin06Elementary} and \cite{Yang14_Peripheral} for corrections).
We proceed by considering $T$ as a splitting over infinite, finitely generated $\widehat{\mathbb{P}}$--parabolic subgroups relative to $\widehat{\mathbb{P}}$.

Consider the tree of cylinders $T_c$ of $T$.
Since $T_c$ is equal to its own tree of cylinders, two cylinders are adjacent in $T$ if and only if the corresponding cylinders are adjacent in $T_c$.
Let $v(e)\in \mathcal{V}$ denote the cylinder whose equivalence class contains the edge $e$.
For each oriented edge $\vec{e}$, the tree of cylinders $T_c$ is the union of two subtrees $S(\vec{e}\,)$ and $S(-\vec{e}\,)$ with intersection $v(e)$ that have the same set of $\mathcal{W}$--vertices as the halftrees $T(\vec{e}\,)$ and $T(-\vec{e}\,)$ respectively.
It follows from Proposition~\ref{prop:TreeOfSpaces} that in $\boundary(G,\widehat{\mathbb{P}})$ the closed halfspaces
\[
   H(\vec{e}\,) = \overline{\bigcup_{v \in \mathcal{W}(S(\vec{e}\,))} \!\!\!\!\Lambda G_v}
   \qquad\text{and}\qquad
   H(-\vec{e}\,) = \overline{\bigcup_{v \in \mathcal{W}(S(-\vec{e}\,))} \!\!\!\!\Lambda G_v}
\]
satisfy $H(\vec{e}\,)\cup H(-\vec{e}\,) = \boundary(G,\widehat{\mathbb{P}})$ and $H(\vec{e}\,)\cap H(-\vec{e}\,) =\Lambda G_{v(e)} = \Lambda G_e$.  Therefore the corresponding open halfspaces are disjoint open sets.

To see that each open halfspace $U(\vec{e}\,)$ of $\boundary(G,\widehat{\mathbb{P}})$ is nonempty, observe that the minimal tree $T$ cannot contain a vertex of valence one, so the halftree $T(\vec{e}\,)$ is unbounded.
As $T_c$ is acylindrical, it follows from Lemma~\ref{lem:acylindrical} that $T$ is also acylindrical.
Thus $T(\vec{e}\,)$ contains more than one cylinder.  In particular it contains at least one vertex with nonelementary stabilizer, which implies that $U(\vec{e}\,)$ is nonempty.

The preimages of the open halfspaces under the natural quotient $\boundary(G,{\mathbb{P}}) \to \boundary(G,\widehat{\mathbb{P}})$ are the open halfspaces of $\boundary(G,\mathbb{P})$ corresponding to $e$, which are again disjoint, nonempty, and open.
\end{proof}

\begin{corollary}
\label{cor:ConnectedHalfspaces}
Let $(G,\bbp)$ be relatively hyperbolic with connected boundary $M=\boundary(G,\bbp)$.  Let $T$ be a minimal $(\mathbb{E}_1,\bbp)$--tree.
Then the closed halfspaces $H(\vec{e}\,)$ and $H(-\vec{e}\,)$ of $M$ corresponding to any oriented edge $\vec{e}$ of $T$ are connected.
\end{corollary}

\begin{proof}
Recall that if the boundary is connected, then each peripheral subgroup must be infinite (see Theorem~\ref{thm:BoundaryConnectedness}).
By Proposition~\ref{Prop: separation}, the closed halfspaces $H(\vec{e}\,)$ and $H(-\vec{e}\,)$ associated to $\vec{e}$ are closed $G_e$--invariant sets with union $M$ and intersection $\Lambda G_e$.
Consider the case that $G_e$ is parabolic with $\Lambda G_e = \{\eta\}$.
Suppose $H(\vec{e}\,) = A \cup B$ for disjoint sets $A$ and $B$ that are closed in $H(\vec{e}\,)$, and hence also in $M$. If $\eta \in A$ then $B \subseteq U(\vec{e}\,)$.  Therefore $M\setminus B = A \cup H(-\vec{e}\,)$ is a closed set of $M$.  Since $M$ is connected, we must have $B=\emptyset$, so that $H(\vec{e}\,)$ is connected.

The nonparabolic, $2$--ended case follows by an argument nearly identical to the proof of \cite[Thm.~0.1]{Bow99_Connectedness}.
\end{proof}

\begin{proposition}
\label{prop:PieceInVertexSpace}
Let $(G,\bbp)$ be relatively hyperbolic with each $P \in \mathbb{P}$ infinite. Let $T$ be a minimal $(\mathbb{E}_1,\bbp)$--tree on which $G$ acts. 
Suppose $C\subseteq \relbndry$ is a subset such that for each edge $e$, the set $C$ lies in one of the open halfspaces corresponding to $e$.
If $C$ contains more than one point then $C$ is contained in $\Lambda G_v$ for some vertex $v$ of $T$.
\end{proposition}


\begin{proof}
Let $\sigma$ be the set of all oriented edges $\vec{e}$ such that the open halfspace $U(\vec{e}\,)$ contains $C$.
We think of $\sigma$ as orienting the edges ``towards $C$.''
Note that two such orientations do not point away from each other, since the corresponding open halfspaces both contain $C$ and thus must have nonempty intersection.
In particular, at each vertex $v \in T$ at most one incident edge can be oriented away from $v$.

Suppose there exists a vertex $v$ that is contained in $T(\vec{e}\,)$ for all $\vec{e} \in \sigma$.
The proof of Proposition~\ref{Prop: separation} shows that in the quotient boundary $\boundary(G,\widehat{\mathbb{P}})$ the intersection of closed halfspaces $\bigcap_{\vec{e}\in\sigma} H(\vec{e}\,)$ is equal to $\Lambda G_v$.
Pulling back via the natural map $\boundary(G,\mathbb{P}) \to \boundary(G,\widehat{\mathbb{P}})$ gives the same conclusion in $\boundary(G,\mathbb{P})$.  Therefore $C \subseteq \Lambda G_v$.

On the other hand, if $T$ contains no such vertex, then $\sigma$ contains a ``ray,'' \emph{i.e.}, a sequence of oriented edges $\vec{e}_1,\vec{e}_2,\dots$ such that $\tau(\vec{e}_i)=\iota(\vec{e}_{i+1})$ and $\iota(\vec{e}_i)\ne\tau(\vec{e}_{i+1})$.
The proof of Proposition~\ref{Prop: separation} shows that in $\boundary(G,\widehat{\mathbb{P}})$
the intersection of the closed halfspaces $H(\vec{e}_i)$ is a single ideal point.
Pulling back via the quotient map $\boundary(G,\mathbb{P}) \to \boundary(G,\widehat{\mathbb{P}})$, we get the same conclusion in $\boundary (G,\mathbb{P})$.
This conclusion contradicts the fact that $C$ contains more than one point and is contained in this intersection.
\end{proof}

\section{Parabolic cut pairs and the proof of \cite[Thm.~4.6]{Grof13}}
\label{sec:Levitt}

This section contains a discussion of two mistakes involving results in \cite[Thm.~4.6]{Grof13}. The first mistake was discovered by the second author and Genevieve Walsh in work related to \cite{HruskaWalsh_Parabolic}.
The authors learned about the second mistake through personal conversation with Brian Bowditch.

The first mistake is that the statement of \cite[Thm.~4.6]{Grof13} depends on an incorrect assertion of Papasoglu--Swenson \cite[\S 5]{PS06} regarding the cut point and cut pair structure of an arbitrary Peano continuum $M$.
As explained below, this assertion is false in general, even in the special case that $M$ is the Bowditch boundary of a relatively hyperbolic pair (as in \cite[Thm.~4.6]{Grof13}).
In particular, \cite[Thm.~4.6]{Grof13} makes an assertion about an object that need not exist under the given hypotheses.
See Example~\ref{exmp:ParabolicCutPairs} for a counterexample.

In \cite[\S 5]{PS06}, Papasoglu--Swenson suggest that one might construct a topological real tree $T$ encoding simultaneously the
cut point structure of a Peano continuum $M$ and the cut pair structure of each of its cyclic elements.  They first construct a canonical topological $\mathbb{R}$--tree $T_0$ dual to the family of cut points of $M$ (see \cite[\S 3]{PS06}); each cyclic element $C$ of $M$ gives rise to a point $p_C$ of $T_0$.
They then ``blow up'' $T_0$ by replacing each point $p_C$ with a copy of a canonical $\mathbb{R}$--tree $T_C$ dual to the family of all inseparable cut pairs of $C$, regardless of exactness (see \cite[\S 4]{PS06}).  In the terminology of \cite[App.~A]{GL_JSJ} the blow-up is a refinement and the map $T \to T_0$ is a collapse.
Papasoglu--Swenson assert that such a refinement $T$ exists with the property that any group $G$ of homeomorphisms of $M$ has an induced action on $T$.

Unfortunately this assertion turns out to be incorrect, since one can only refine $T_0$ equivariantly if each cyclic element $C$ of $M$ has the following property:
for each component of $M-C$ with stabilizer $Q \le \Stab_G(C)$ there exists an attaching point $x_Q\in \overline{T}_C = T_C \cup \boundary T_C$ fixed by $Q$.
It is shown (correctly) in \cite[\S 5]{PS06} that the cyclic element $C$ must contain a point $y_Q$ fixed by $Q$.
The mistake in \cite[\S 5]{PS06} is that in general the action of $Q$ on $\overline{T}_C$ need not have a fixed point.
The construction of $T_C$ in \cite[\S 4]{PS06} associates to $y_Q$ a subtree $T_Q \subseteq T_C$ dual to the family of all inseparable cut pairs that contain $y_Q$.
If $y_Q$ is contained in many different inseparable cut pairs, there need not exist a single cut pair of $C$ stabilized by $Q$.  Consequently $\overline{T}_C$ need not contain a point fixed by $Q$.
(A counterexample is given below.)
Thus there does not exist a tree $T$ with a $G$--action that collapses onto $T_0$ such that the preimage of each point of type $p_C$ is $\Stab(C)$--equivariantly homeomorphic to the tree $\overline{T}_C$.
The main result of \cite[\S 5]{PS06} asserts that such a tree exists, but this assertion is not true in general.

The following example based on a construction of Gaboriau--Paulin \cite[\S 3.4, Example~1]{GaboriauPaulin01} shows that the required fixed point $x_Q$ does not exist in general.
See \cite[Prop.~3.4]{HruskaWalsh_Parabolic} for a more detailed discussion of this example.

\begin{example}[Parabolic cut pairs]
\label{exmp:ParabolicCutPairs}
Consider the group $H = F_1 * F_2 *F_3$, where each $F_i=\Z/3\Z$, and let $\mathbb{Q} = \{F_1*F_2,\ F_1*F_3,\ F_2*F_3\}$.
Since $\mathbb{Q}$ is an almost malnormal family, $(H,\mathbb{Q})$ is relatively hyperbolic by a theorem of Bowditch \cite[Thm.~7.11]{BowditchRelHyp}. Furthermore, $H$ does not split as a graph of groups relative to $\mathbb{Q}$ by Serre's Lemma \cite[\S I.6.5, Cor.~2]{Serre}.
By Theorem~\ref{thm:BoundaryConnectedness}, the boundary $C$ is connected and locally connected and every cut point of $C$ is a parabolic point.  Since $H$ does not split relative to $\mathbb{Q}$, it follows from \cite[Thm.~9.2]{Bow01_Peripheral} that $C$ has no cut points.
By \cite{Haulmark_LocalCutPt}, all local cut points of $C$ are parabolic.
Thus every cut pair of $C$ consists of a pair of parabolic points.

In \cite[Prop.~3.4]{HruskaWalsh_Parabolic}, it is shown that $C$ does have a nonempty family of parabolic inseparable cut pairs, and the inseparable cut pair tree $T_C$ of \cite[\S 4]{PS06} is a simplicial tree $T_C$ in which all vertices have valence $3$.
Furthermore, the group $H$ acts minimally on this tree, each edge has trivial stabilizer, and the stabilizer of each vertex is either trivial or of order three.
On the other hand, each virtually free peripheral subgroup $Q\in \mathbb{Q}$ acts properly and cocompactly on a subtree $T_Q$ such that action of $Q$ on $\overline{T}_Q$ has no fixed point.
In particular, since $Q$ does have a fixed point in $C$, we note that there does not exist an $H$--equivariant map $C \to \overline{T}_C$.

Given the existence of the relatively hyperbolic pair $(H,\mathbb{Q})$, one can easily create relatively hyperbolic pairs such that the cut-point/cut-pair tree construction fails, as follows.
Let $G$ be the fundamental group of the following graph of groups:
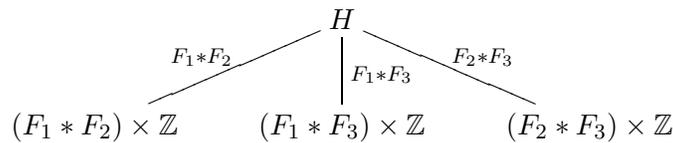
\[
\xymatrix{
  &  H \ar@{-}[ld]_{F_1*F_2} \ar@{-}[d]^{F_1*F_3} \ar@{-}[rd]^{F_2*F_3} & \\
(F_1*F_2)\times \Z & (F_1*F_3)\times \Z & (F_2*F_3)\times\Z
}
\]
By the combination theorem of Dahmani \cite{Dahm03}, the group $G$ is hyperbolic relative to
\[
   \mathbb{P} = \bigl\{ (F_1*F_2)\times \Z, \ (F_1*F_3)\times \Z, \ (F_2*F_3)\times \Z \bigr\}.
\]
Recall that $H$ does not admit a splitting relative to $\mathbb{Q}$.
Since $C$ is locally connected, 
a result of Bowditch \cite{Bow01_Peripheral} implies that the boundary $M = \boundary (G,\mathbb{P})$ is locally connected, the nontrivial cyclic elements of $M$ are each homeomorphic to $C$, and $M$ has a simplicial cut point tree $T_0$ equal to the Bass--Serre tree of the given splitting of $G$.

In this situation, the refinement claimed to exist in \cite{PS06} cannot exist because the edge stabilizers of $T_0$ do not have fixed points in $\overline{T}_C$, so that the required attaching point $p_Q \in \overline{T}_C$ does not exist.

We remark that the group $G$ constructed above is finitely presented and one ended, has no infinite torsion subgroup, and has the property that all peripheral subgroups are non--relatively hyperbolic. 
Thus it satisfies the hypothesis of \cite[Thm.~4.6]{Grof13}.
\end{example}

In some cases, the cut-point/cut-pair topological real tree $T$ may exist (in the case that all required attaching points exist).
A second error in the proof of \cite[Thm.~4.6]{Grof13} involves a mistaken application of a theorem due to Levitt.
The main theorem of Levitt's article \cite{Levitt98} states that whenever a finitely presented group $G$ admits a nontrivial non-nesting action by homeomorphisms on a topological real tree $T$, the group $G$ admits a nontrivial isometric action on a metric $\R$--tree $T_m$ such that the stabilizer of each arc in $T_m$ also stabilizes an arc in $T$.

The proof of \cite[Thm.~4.6]{Grof13} depends on an unjustified assertion that the metric tree $T_m$ is a $G$--equivariant quotient of the topological tree $T$.
A careful reading of Levitt's construction reveals that in general there need not be a $G$--equivariant map $T\to T_m$.
Indeed if one applies Levitt's construction to an isometric action that is nongeometric in the sense of \cite{LevittPaulin97} then such an equivariant map cannot exist because the given action does not admit an exact resolution (see \cite{LevittPaulin97} for details).
Based on Groff's analysis, it is unclear whether the Papasoglu--Swenson cut pair/cut point tree for a relatively hyperbolic group is geometric in the above sense (in the case when it exists at all).

In the next section we introduce a different tree dual to the family of exact inseparable cut pairs and cut points of the boundary.  The analysis of this alternative tree involves neither Rips Theory nor the theorem of Levitt mentioned above.

\section{The exact cut pair/cut point tree}
\label{sec: simplicial}

Suppose $(G,\bbp)$ is relatively hyperbolic with connected boundary $M$.
In this section we introduce a simplicial tree dual to the family of cut points and inseparable exact cut pairs in $M$, called the exact cut pair/cut point tree of $M$.

According to Proposition~\ref{Prop: separation}, each elementary splitting of a relatively hyperbolic groups gives rise to a family of cuts of the boundary; that is, the boundary is cut into pieces by the limit sets of the edge groups.
More precisely, a parabolic edge group gives rise to a cut point, and a nonparabolic $2$--ended edge group gives rise to a cut pair, which is always an exact cut pair by Lemma~\ref{lem: 2ended are exact} together with Corollary~\ref{cor:SplittingExact} in this section.

In this section we prove Theorem~\ref{Thm: construction of T}, which establishes the existence of the exact cut pair/cut point tree $T$ for any connected Bowditch boundary. 
We also establish Proposition~\ref{prop: vertex types}, which classifies the possible vertex stabilizers of the induced action on the tree $T$.

As in Section~\ref{sec: trees separate}, a pinching argument due to Dahmani \cite{Dahm03} may be used to convert loxodromic cut pairs into parabolic cut points.
Thus the construction of the tree $T$ is reduced to the study of cut points in the  boundary, which is known by the following result, which combines Theorem~\ref{thm:BoundaryConnectedness} with \cite[Thm.~9.2]{Bow01_Peripheral}.

\begin{theorem}[Bowditch, Dasgupta]
\label{thm: peripheral splittings}
Let $(G,\bbp)$ be relatively hyperbolic with connected boundary $M=\bndry(G,\bbp)$.
Let $T$ be the bipartite graph with vertex set $\mathcal{V} \sqcup \mathcal{W}$, where $\mathcal{V}$ is the set of cut points and $\mathcal{W}$ is the set of nontrivial cyclic elements of $M$.  Two vertices $v \in \mathcal{V}$ and $w \in \mathcal{W}$ are connected by an edge in $T$ if and only if the cut point $v$ is contained in the cyclic element $w$.

The graph $T$ is a JSJ tree for splittings of $G$ over parabolic subgroups relative to $\bbp$.
The edges of $T$ lie in finitely many $G$--orbits and the stabilizer of each edge is finitely generated.
\end{theorem}

We first need a proposition which allows us to apply the results of Section~\ref{sec:NoCutPoint} to the nontrivial cyclic elements of $\partial(G,\bbp)$.
According to the following result, the notions of cut pair and exactness have the same meaning when considering pairs of points in $\relbndry$ and in a nontrivial cyclic element.

\begin{proposition}[reduction]
\label{prop: reduction}
Let $(G,\bbp)$ be relatively hyperbolic with connected boundary, and let $C$ be a nontrivial cyclic element of $\relbndry$. Then the following hold.
\begin{enumerate}
    \item 
    \label{prop: red1}
    The set $C$ is connected and locally connected.
    \item
    \label{prop: red2}
    The stabilizer $H$ of $C$ is hyperbolic relative to a family $\mathbb{O}$ of representatives of the conjugacy classes of infinite subgroups of the form $H\cap gPg^{-1}$ where $g\in G$ and $P\in\bbp$. Additionally, the boundary $\bndry(H,\mathbb{O})$ is $H$--equivariantly homeomorphic to $C$.
    \item 
    \label{prop: red3}
    If $\xi$ is a non-parabolic local cut point of $C$, then the valence of $\xi$ in $C$ is the same as the valence of $\xi$ in $\relbndry$.
    \item
    \label{prop: red4}
    Assume that $\xi,\eta\in C$ are not parabolic points. The pair $\{\xi,\eta\}$ is a cut pair of $C$ if and only if $\{\xi,\eta\}$ is a cut pair of $\relbndry$. 
    Furthermore, if a pair $\{\xi,\eta\}$ is exact in $C$ it is exact in $\relbndry$.
    \item 
    \label{prop: red5}
    A subset $\nu\subset C$ is a necklace in $C$ if and only if $\nu$ is a necklace of $\relbndry$.
\end{enumerate}
\end{proposition}

\begin{proof}
Since $\relbndry$ is connected and locally connected, each of its nontrivial cyclic elements is as well by \cite[\S 7]{Bow01_Peripheral}. Conclusion~(\ref{prop: red2}) is due to Bowditch \cite[Thm.~1.3]{Bow01_Peripheral}.  For Conclusions (\ref{prop: red3}), (\ref{prop: red4}), and~(\ref{prop: red5}), see \cite[\S 3]{Haulmark_LocalCutPt}.
\end{proof}

Corollary~\ref{cor:ConnectedHalfspaces} does not depend on the local connectedness of the boundary.
This result is used by Dasgupta in his proof that a connected Bowditch boundary is always locally connected (see Theorem~\ref{thm:BoundaryConnectedness}).
Applying Theorem~\ref{thm:BoundaryConnectedness} gives the following stronger conclusion.

\begin{corollary}
\label{cor:SplittingExact}
Let $(G,\mathbb{P})$ be relatively hyperbolic with connected boundary $M$.
Let $T$ be a minimal $(\mathbb{E},\mathbb{P})$--tree.
If an edge $e$ has nonparabolic $2$--ended stabilizer $G_e$, then its limit set $\Lambda G_e$ is an exact cut pair of $M$.
\end{corollary}

\begin{proof}
The conclusion of Corollary~\ref{cor:ConnectedHalfspaces} would imply that $\Lambda G_e$ is a cut pair if we knew that neither point of $\Lambda G_e$ is a cut point.
But Theorem~\ref{thm:BoundaryConnectedness} gives that the points of $\Lambda G_e$ cannot be cut points, since they are not parabolic.
If $M$ has no cut point, then the cut pair $\Lambda G_e$ is exact by Lemma~\ref{lem: 2ended are exact}.
The general case, in which $M$ has cut points follows from Proposition~\ref{prop: reduction}.
\end{proof}

\begin{defn}[Pieces]
Let $Z$ be the union of all global cut points and inseparable exact cut pairs of $\bndry(G,\bbp)$.  Consider the equivalence relation on $\bndry(G,\bbp)\setminus Z$ where two points are related if they cannot be separated by a cut point or an inseparable exact cut pair. The closure of an equivalence class containing at least two points is a {\it piece}. There are two types of pieces: those that can be separated by an exact cut pair and those that cannot. A piece which cannot be separated by an exact cut pair is {\it rigid}. The stabilizer in $G$ of a rigid piece is a {\it rigid subgroup}.
\end{defn}

\begin{theorem}
\label{Thm: construction of T}
Let $(G,\bbp)$ be relatively hyperbolic with connected boundary $M=\bndry(G,\bbp)$. 
Consider the bipartite graph $T$ with vertex set $Z\sqcup \Pi$ where $Z$ is the set of all global cut points and inseparable exact cut pairs and $\Pi$ is the collection of pieces of $M$. Two vertices $z\in Z$ and $\pi\in\Pi$ are connected by an edge in $T$ if $z$ is contained in $\pi$.

Then $T$ is a simplicial $(\mathbb{E},\bbp)$--tree on which $G$ acts, where $\mathbb{E}$ is the collection of elementary subgroups of $(G,\bbp)$. 
\end{theorem}

The tree $T$ constructed in the theorem is the \emph{exact cut pair/cut point tree} for $M$.  

\begin{proof}
By Theorem~\ref{thm:BoundaryConnectedness}, all cut points of $\relbndry$ are parabolic and $M$ is locally connected. 
For each nontrivial cyclic element $C$ of $M$, let $H$ be the stabilizer of $C$, and let $\mathbb{O}$ be a set of representatives of the conjugacy classes of infinite subgroups of the form $H\cap gPg^{-1}$ such that $g\in G$ and $P\in\bbp$.
By Theorem~\ref{thm: peripheral splittings} and Proposition~\ref{prop: reduction}, the cyclic element $C$ is connected, locally connected, without cut points, and homeomorphic to $\bndry(H,\mathbb{O})$.

Since there are only finitely many conjugacy classes of peripheral subgroups, Theorem~\ref{thm: peripheral splittings} implies there are only finitely many orbits of nontrivial cyclic elements in $\bndry(G,\bbp)$.
Each cut pair of $M$ is contained in a nontrivial cyclic element $\boundary(H,\mathbb{O})$ (see Lemma~\ref{lem:CutPair}). By Proposition~\ref{prop: only finitely many} a nontrivial cyclic element $C=\bndry(H,\mathbb{O})$ contains only finitely many $H$--orbits of inseparable exact cut pairs, and the stabilizer of each is a maximal non-parabolic $2$--ended subgroup of $H$.
In particular, an inseparable exact cut pair consists of non-parabolic points.
Thus $\boundary(G,\mathbb{P})$ contains only finitely many $G$--orbits of exact inseparable cut pairs, and the $G$--stabilizer of each is a maximal non-parabolic $2$--ended subgroup of $G$.

As in the proof of Proposition~\ref{Prop: separation}, we form $\widehat{\mathbb{P}}$ by choosing representatives of the finitely many $G$--orbits of exact inseparable cut pairs, and adding their $G$--stabilizers to $\mathbb{P}$.
Then $(G,\widehat{\mathbb{P}})$ is relatively hyperbolic with connected boundary $\boundary(G,\widehat{\mathbb{P}})$ formed from $M$ by collapsing to a point each exact inseparable cut pair \cite{Dahm03,Osin06Elementary,Yang14_Peripheral}.
As observed in \cite[Lem.~5.2]{Haulmark_LocalCutPt}, inseparability implies that the resulting pinched points are cut points.

Observe that the tree $T$ given by applying Theorem~\ref{thm: peripheral splittings} to $(G,\widehat{\bbp})$ may be considered as an $(\mathbb{E},\bbp)$--tree, with respect to the original peripheral structure $\mathbb{P}$.
Indeed, $q$ induces a bijection between the set $Z \subset \relbndry$ and the set of cut points of $\boundary(G,\widehat{\mathbb{P}})$ and also a bijection between the set of pieces of $\relbndry$ and the set of nontrivial cyclic elements of $\bndry(G,\widehat{\bbp})$. 
Since these bijections respect adjacency, we see that $T$ is the desired $(\mathbb{E},\mathbb{P})$--tree.
\end{proof}

A consequence of this construction is a structure result that classifies the types of vertex stabilizers of the exact cut tree. 
First we will need a lemma concerning pieces.

\begin{lemma}
\label{lemma: piece necklace}
Every nonrigid piece of $\bndry(G,\bbp)$ is a necklace.
\end{lemma}

\begin{proof}
Let $Y$ be a piece of $\bndry (G,\bbp)$, which by definition is not separated by any inseparable exact cut pair of $M$.
Suppose $Y$ is separated by an exact cut pair $\{\zeta,\xi\}$ of $M$. Then such a cut pair cannot be inseparable in $M$.
Let $C$ be the cyclic element of $M$ containing the piece $Y$.
Then $\{\zeta,\xi\}$ is a cut pair of $C$ that is not inseparable by Proposition~\ref{prop: reduction}(\ref{prop: red4}).
Thus Lemma~\ref{lemma: cut not separated by M(3+)} implies that $\{\zeta,\xi\}$ cannot be an isolated exact cut pair of $C$.
It follows that $\{\zeta,\xi\}$ must be contained in a $\sim$--class $N$ of $C$ containing at least three elements. We will show that $\overline{N}$ is equal to $Y$.

First assume that $C$ is homeomorphic to the circle.  Then the continuum $C$ consists of a single $\sim$--class of $C$.
In this case, any necklace $N$ contained in $C$ must equal $C$ itself.  Therefore $\overline{N}=Y$ since $N\subseteq \overline{N} \subseteq Y \subseteq C$.

Now assume that $C$ is not homeomorphic to $S^1$ and let $y\in Y\setminus\overline{N}$.
Proposition~\ref{prop: reduction}(\ref{prop: red5}) together with Proposition~\ref{prop:Bivalent} give that $\overline{N}$ is homeomorphic to a Cantor set, and each jump is inseparable and is stabilized by a loxodromic element. Thus by Lemma~\ref{lem: 2ended are exact} each jump is exact.  
Since $C$ is a cyclic element and $\overline{N}\subset Y\subset C$, Lemma~2.3 of Bowditch implies that each component of $C\setminus\overline{N}$ is separated from $\overline{N}$ by a jump. 
Thus $y$ is separated from $N$ by an inseparable exact cut pair, a contradiction.
\end{proof}

Our ultimate goal is to show that the exact cut pair/cut point tree $T$ is equal to the canonical JSJ tree of cylinders constructed by Guirardel--Levitt \cite{GL_TreesCyl} for elementary splittings relative to peripheral subgroups.
In the following proposition we show that $T$ has certain key features in common with the JSJ tree of cylinders. In Proposition~\ref{prop: cmbd tree is JSJ} we exploit the relationship between vertex types and topological features of Bowditch boundary to show that $T$ is JSJ tree.

\begin{proposition}
\label{prop: vertex types}
Let $T$ be the exact cut pair/cut point tree constructed in Theorem~\ref{Thm: construction of T}.
Every vertex stabilizer of $T$ is one of the following types:
\begin{enumerate}
    \item peripheral subgroup
    \item non-parabolic $2$--ended
    \item quadratically hanging with finite fiber, or
    \item rigid.
\end{enumerate}
\end{proposition}

\begin{proof}
Let $(G,\bbp)$ be the given relatively hyperbolic structure, and let $(G,\widehat \bbp)$ be the augmented relatively hyperbolic structure constructed in the proof of Theorem~\ref{Thm: construction of T}.
Under the quotient map  $\bndry(G,\bbp)\rightarrow\bndry(G,\widehat \bbp)$ each inseparable exact cut pair is collapsed to a cut point.
Each cut point of $\bndry(G,\bbp)$ maps to a cut point of $\bndry(G,\widehat \bbp)$, and each piece and each necklace maps to a nontrivial cyclic element.
Furthermore the quotient induces a bijection from the ideal points of $\bndry(G,\bbp)$ to the ideal points of $\bndry(G,\widehat \bbp)$.

Ideal points do not separate the Bowditch boundary (see the proof of \cite[Cor.~3.2]{Haulmark_LocalCutPt}), and by Lemma~\ref{lemma: piece necklace} necklaces and rigid components are the only two types of pieces. So $T$ only has vertices corresponding to cut points, inseparable exact cut pairs, necklaces, and rigid components.

We need only show that the stabilizer of a necklace $\nu$ corresponding to a vertex $v$ is quadratically hanging with finite fiber. A similar result is established by Groff \cite[Prop.~7.2]{Grof13} using slightly different definitions. Let $H=\Stab(\nu)$. 
By Proposition~\ref{prop:Bivalent}, the necklace $\nu$ contains only finitely many $H$--orbits of jumps.
Let $\mathbb{J}$ be a finite set of subgroups representing the conjugacy classes of stabilizers of jumps in the action of $H$ on $\nu$. Since $\nu$ is a piece of $\relbndry$, it is a nontrivial cyclic element of $\partial(G,\widehat \bbp)$. Thus by Theorem~\ref{thm: peripheral splittings} and Proposition~\ref{prop: reduction} we may apply Proposition~\ref{prop:Bivalent} to conclude that $\nu$ is a cyclically ordered Cantor set if $\mathbb{J}$ is nonempty or a circle if $\mathbb{J}$ is empty.  Furthermore, $(H,\mathbb{O})$ is relatively hyperbolic, where $\mathbb{O}$ is a set of representatives for the conjugacy classes of maximal parabolic subgroups in $H$.
As in Theorem~\ref{Thm: construction of T},
it follows that $(H,\mathbb{J} \cup \mathbb{O})$ is relatively hyperbolic, and its Bowditch boundary is homeomorphic to $S^1$.

In particular, as explained in \cite{BowditchRelHyp}, the action of $H$ on $S^1$ is a convergence group action; in other words, for each sequence $(h_i)$ of distinct elements of $H$ there exist $\zeta,\xi \in S^1$ such that, after passing to a subsequence, $h_i\big| S^1\setminus\{\zeta\} \to \xi$ uniformly on compact sets.
The kernel $F$ of this action is finite, and the quotient $H/F$ is a faithful convergence group acting on $S^1$.
It follows from \cite{T88,Gabai92,CJ94} that the action of $H/F$ on $S^1$ extends to a faithful, proper, isometric action on the hyperbolic plane.

Since $H/F$ has limit set equal to $S^1$, it is the fundamental group of a complete, finite area hyperbolic $2$--orbifold $\Sigma$ with finitely many cusps (see \cite[\S 4.5]{Katok_Fuchsian}).
Truncating the cusps, we may assume that $\Sigma$ is a compact hyperbolic $2$--orbifold with boundary. Each member of the family $\mathbb{J} \cup \mathbb{O}$ is the preimage in $H$ of a boundary subgroup of $\Sigma$.
Every parabolic point of $\boundary(G,\mathbb{P})$ contained in $\nu$ is stabilized by a conjugate of a member of $\mathbb{O}$.  Since every inseparable exact cut pair of $\boundary(G,\mathbb{P})$ contained in $\nu$ is a jump of $\nu$, it must be stabilized by a conjugate of a member of $\mathbb{J}$.
Therefore each edge adjacent to $v$ corresponds to a boundary component of $\Sigma$.
\end{proof}

\section{Proof of the main theorem}
\label{sec: last section}

In this section we complete the proof of Theorem~\ref{thm: main thm} by showing that the exact cut pair/cut point tree $T$ is a JSJ tree with quadratically hanging flexible vertex stabilizers. More precisely, we prove:

\begin{proposition}
\label{prop: cmbd tree is JSJ}
Let $(G,\bbp)$ be a relatively hyperbolic group with connected boundary $M=\bndry(G, \bbp)$.  The exact cut pair/cut point tree $T$ for $\relbndry$ is a JSJ tree for splittings over elementary subgroups relative to peripheral subgroups.
Moreover, the flexible vertex stabilizers of $T$ are quadratically hanging with finite fiber.
\end{proposition}

\begin{proof}
By Theorem~\ref{Thm: construction of T}, the tree $T$ is an  $(\mathbb{E},\bbp)$--tree. We first show that $T$ is universally elliptic; \emph{i.e.}, each edge stabilizer $H$ of $T$ acts elliptically on every other $(\mathbb{E},\bbp)$--tree.
Recall that if $T_1,T_2$ are any $(\mathbb{E},\bbp)$--trees such that $T_1$ dominates $T_2$, then by definition if $H$ acts elliptically on $T_1$ it also acts elliptically on $T_2$.
Since every $(\mathbb{E},\bbp)$--tree is dominated by one with finitely generated edge stabilizers \cite[Cor.~2.25]{GL_JSJ}, it suffices to show that $H$ acts elliptically on each $(\mathbb{E},\bbp)$--tree $S$ that has finitely generated edge stabilizers.
Furthermore parabolic subgroups of $G$ are universally elliptic by definition.  
Thus we may also assume that $H$ is a nonparabolic $2$--ended edge group.

Suppose $H$ stabilizes an inseparable exact cut pair $\{\zeta,\xi\}\subset\relbndry$. Let $S$ be any $(\mathbb{E},\bbp)$--tree with finitely generated edge groups.  
Recall that the limit sets of any loxodromic or parabolic elements $g$ and $h$ either coincide or are disjoint by \cite[Thm.~2G]{Tuk94}.
If $\{\zeta,\xi\}$ coincides with the limit set of an edge group $G_e$ of $S$, then $H$ and $G_e$ are co-elementary, since $H$ and $G_e$ each have finite index in the stabilizer of $\{\zeta,\xi\}$ by \cite[Thm.~2I]{Tuk94}.  Then $H$ has a finite index subgroup that acts elliptically on $S$, so $H$ also acts elliptically.
It suffices to assume that $\{\zeta,\xi\}$ is disjoint from the limit sets of all edge groups.
By Proposition~\ref{Prop: separation}, any parabolic edge group of $S$ has limit set a cut point of $M$. Similarly by Corollary~\ref{cor:SplittingExact}, any $2$--ended nonparabolic edge group $G_e$ has limit set an exact cut pair.
In particular, the limit set of $H$ is not separated by the separations of $M$ corresponding to any edge group of $S$.
By Proposition~\ref{prop:PieceInVertexSpace} the pair $\{\zeta,\xi\}$ is contained in $\Lambda G_v$ for some vertex $v$ of $S$.  This vertex $v$ must be fixed by $H$, so $H$ acts elliptically on $S$.

We now need to show that $T$ dominates every other universally elliptic tree $S$.
In other words, each vertex group $H$ of $T$ acts elliptically on every universally elliptic tree $S$.
If $H$ is parabolic, then it acts elliptically on the $(\mathbb{E},\mathbb{P})$--tree $S$ by definition.  If $H$ is $2$--ended and stabilizes an inseparable exact cut pair, then it acts elliptically on $S$ by the argument of the previous paragraph.
If $H$ is quadratically hanging with finite fiber, then $H$ acts elliptically on every universally elliptic tree by a result of Guirardel--Levitt on the existence of incompatible splittings of quadratically hanging groups \cite[Thm.~5.27]{GL_JSJ}.

By Proposition \ref{prop: vertex types} the only remaining case is a vertex group $H$ that is rigid.
For each rigid group $H$, by definition there exists a set $A \subset \relbndry$ stabilized by $H$ that is not separated by any cut point or exact cut pair.  Proposition~\ref{prop:PieceInVertexSpace} gives that such a set $A$ is contained in $\Lambda G_v$ for some vertex $v$ of $S$, and thus that $H$ fixes the vertex $v$ of $S$.  So $H$ acts elliptically on $S$.
\end{proof}

\begin{proof}[Proof of Theorem~\ref{thm: main thm}]
Trees that are equivalent in the sense of Section~\ref{sec:JSJ} have the same tree of cylinders by Theorem~\ref{thm:TreeOfCylinders}.
To show that the exact cut pair/cut point tree $T$ is equal to the JSJ tree of cylinders of \cite{GL_TreesCyl}, it suffices to show that it is equal to its own tree of cylinders.
By definition, $T$ is a bipartite graph with vertex set $Z \sqcup \Pi$.  Each vertex of $Z$ has an elementary stabilizer, and each vertex of $\Pi$ has a stabilizer that is either quadratically hanging with finite fiber or rigid.
Each vertex $v \in Z$ has the property that all edges adjacent to it are equivalent, since their stabilizers are subgroups of the elementary group $\Stab(v)$.
On the other hand, each vertex of $\Pi$ has the property that none of its adjacent edges are co-elementary.
Indeed, distinct neighbors correspond to distinct members of $Z$, the set of all cut points and exact inseparable cut pairs of the boundary.  If $z\ne z' \in Z$ the group generated by $\Stab(z) \cup \Stab(z')$ has a limit set that contains $z \cup z'$.  Thus this limit set either contains at least three distinct points or contains two distinct parabolic points. Such a subgroup cannot be elementary, since in any convergence group action, the limit set of an elementary subgroup contains at most two points, at most one of which can be parabolic (see \cite[\S 2]{Tuk94}).
It follows that $T$ is equal to its own tree of cylinders.
\end{proof}

\bibliographystyle{alpha}
\bibliography{mybib}

\end{document}